\documentclass[10pt,reqno]{amsart}
\usepackage{amsmath,amsfonts,amssymb,amscd,amsthm,amsbsy, epsf,calc,enumerate,cite}
\usepackage{color}
\usepackage{graphicx}
\usepackage{epsfig,esint}
\usepackage{verbatim}
\usepackage{soul}
\definecolor{purple}{rgb}{0.65, 0, 1}
\definecolor{orange}{rgb}{1,.5,0}

\def\e{\epsilon}
\def\R{\mathbb{R}}

\def\II{{\rm I\kern-0.5exI}}
\def\III{{\rm I\kern-0.5exI\kern-0.5exI}}
\def\argmin{\mathrm{argmin}}
\numberwithin{equation}{section}
\newtheorem{theorem}{Theorem}[section]
\newtheorem{remark}[theorem]{Remark}
\newtheorem{lemma}[theorem]{Lemma}
\newtheorem{definition}[theorem]{Definition}
\newtheorem{proposition}[theorem]{Proposition}
\newtheorem{corollary}[theorem]{Corollary}

\newcommand{\comments}[1]{}
\newcommand{\supp}[0]{\mbox{supp}}

\newcounter{DTheoremCounter}
\newcommand{\pme}{(PME-D)$_m$}

\newcounter{DAppCounter}
\newcounter{claimCounter}
\usepackage{tikz}

\usetikzlibrary{patterns}
\usetikzlibrary{decorations.pathreplacing}
\usetikzlibrary{decorations.markings}
\usetikzlibrary{matrix}
\usetikzlibrary{arrows}

\pgfdeclarepatternformonly{north east lines wide}%
   {\pgfqpoint{-1pt}{-1pt}}%
   {\pgfqpoint{10pt}{10pt}}%
   {\pgfqpoint{9pt}{9pt}}%
   {
     \pgfsetlinewidth{0.5pt}
     \pgfpathmoveto{\pgfqpoint{0pt}{0pt}}
     \pgfpathlineto{\pgfqpoint{9.1pt}{9.1pt}}
     \pgfusepath{stroke}
    }
    
\pgfdeclarepatternformonly{north west lines wide}%
   {\pgfqpoint{-1pt}{-1pt}}%
   {\pgfqpoint{10pt}{10pt}}%
   {\pgfqpoint{9pt}{9pt}}%
   {
     \pgfsetlinewidth{0.4pt}
     \pgfpathmoveto{\pgfqpoint{0pt}{9.1pt}}
     \pgfpathlineto{\pgfqpoint{9.1pt}{0pt}}
     \pgfusepath{stroke}
    }    
\definecolor{pink}{rgb}{1, 0, 0.7}
\definecolor{purple}{rgb}{0.65, 0, 1}

\begin{document}
\title{Quasi-static evolution and congested crowd transport}
\author{Damon Alexander, Inwon Kim, and Yao Yao}
\address{Dept. of mathematics, UCLA and UW-Madison} 
\date{}
\begin{abstract}
We consider the relationship between Hele-Shaw evolution with drift, the porous medium equation with drift, and a congested crowd motion model originally  proposed by \cite{mrs2}-\cite{mrs1}.  We first use viscosity solutions to show that the porous medium equation solutions converge to the Hele-Shaw solution as $m \to \infty$ provided the drift potential is strictly subharmonic.   Next, using of the gradient flow structure of both the porous medium equation and the crowd motion model, we prove that the porous medium equation solutions also converge to the congested crowd motion as $m\to\infty$. Combining these results lets us deduce that in the case where the initial data to the crowd motion model is given by a \textit{patch}, or characteristic function, the solution evolves as a patch that is the unique solution to the Hele-Shaw problem.  While proving our main results we also obtain a comparison principle for solutions to the minimizing movement scheme based on the Wasserstein metric, of independent interest.

\end{abstract}
\thanks{D. Alexander and I. Kim are supported by NSF DMS-0970072. Y. Yao is supported by NSF-DMS 1104415}
\maketitle
\section{Introduction} 

Let $\Omega_0$ be a compact set in $\R^d$ with locally Lipschitz boundary, and let $\Phi(x): \R^d\to \R$ be a $C^2$ function which satisfies
$$
\Delta\Phi>0\hbox{ in }\R^d. \leqno (\textbf{A1})
$$
For $\Omega_0$ and $\Phi$ as given above, we consider a function $u: \R^d\to \R$, $u(x,t) \geq 0$ solving the following free boundary problem: 
$$
\left\{\begin{array}{lll}
-\Delta u (\cdot,t) = \Delta \Phi &\hbox { in }& \{u > 0\} ;\\ 
V = -\partial_{\nu} u -\partial_{\nu}\Phi &\hbox{ on }& \partial\{u>0\}.
\end{array}\right.\leqno(P)
$$
Here $\nu_{x,t}$ is the outward normal vector of the set $\Omega_t(u):=\{x: u(x,t) >0\}$ at $x\in \Gamma_t(u):=\partial\Omega_t(u)$, and $V$ denotes the outward normal velocity of $\Omega_t(u)$ at $x\in\Gamma_t(u)$. 

In terms of $u$, $\nu = -\nabla u/|\nabla u|$ and thus one can write down the second condition of $(P)$ as 
$$
u_t = |\nabla u|^2 + \nabla u\cdot\nabla\Phi \quad \hbox{ on } \partial\{u>0\},
$$
given that $|\nabla u| \neq 0$ at the boundary point.  Note that the free boundary velocity $V$ may be positive or negative depending on the behavior of $\Phi$ on $\Gamma(t)$.  Consequently $\Omega_t(u)$ may expand or shrink over time (see Figure 1). Indeed formal calculations based on $(P)$ yield that $\Omega_t$ preserves its volume over time. 
The initial data $u(x,0)= u_0$ is the unique function satisfying 
\begin{equation}\label{initial}
-\Delta u_0 = \Delta\Phi \hbox{ in the interior of } \Omega_0, \quad u_0 = 0 \hbox{ on } \Omega_0^C.
\end{equation}

Note that, due to $(\textbf{A1})$, $u_0$ is positive in $\Omega_0$ and thus \eqref{initial} is well-defined. Still, even starting from a smooth domain $\Omega_0$, the solution of $(P)$ can develop finite-time singularities as its support goes through topological changes such as pinching and merging, and thus it is necessary to consider a notion of weak solutions. We will use the notion of viscosity solutions for $(P)$, see section 2 for definitions and properties of $u$.  Let us mention that the usual variational inequality formulation for weak solutions of Hele-Shaw flow,  introduced by \cite{ej}, does not apply here due to the non-monotonicity of solutions in time variable.
 
\begin{figure}
\begin{tikzpicture}
\draw plot [smooth cycle, tension = 0.8] coordinates {(2,0) (4,2) (6,0) (5,-2) };
\draw plot [smooth cycle, tension = 0.8] coordinates {(-1,0) (0,2) (-2, 1) (-4,2) (-4.4, -1) (-2, -2.5) };
\draw [->, thick, color=blue] (-1,0) to (-0.5,-0.2);
\draw [->, thick, color=blue] (0,2) to (-0.25,1.7);
\draw [->, thick, color=blue] (-4,2) to (-3.65,1.6);
\draw (-3.3,2.3) node {\textcolor{blue}{\large{$V = -\partial_\nu u - \partial_\nu \Phi$}}};
\draw (-3,0) node {\large{$u(\cdot,t)>0$}};
\draw (-2,1.6) node {\large{$u=0$}};

\draw [-latex, very thick] (0,0) -- node[pos = 0.45, yshift = 0.3cm, black] {$t\to\infty$} (1.7,0) ;
\draw (4,0) node {\large{supp~$s(x)$}};
\draw (4,-2.3) node {\large{equilibrium profile}};
\draw (4.2,-2.8) node {\large{$s(x) = (C-\Phi(x))_+$}};

\end{tikzpicture}
\caption{Evolution of the positive phase, converging toward the equilibrium}
\end{figure}
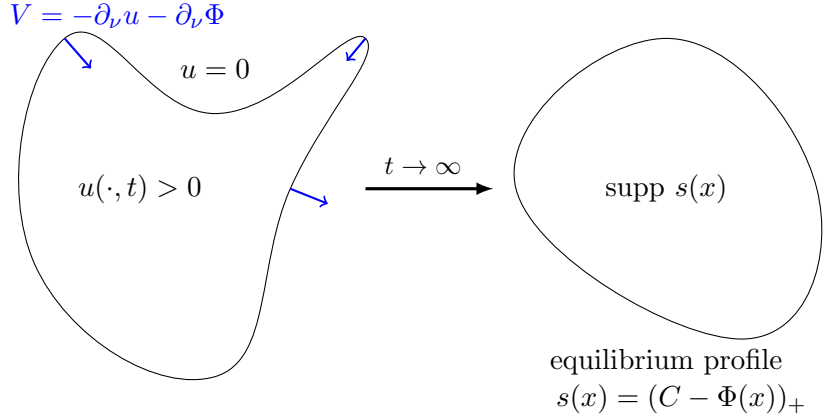

\medskip

In the context of fluid dynamics, the problem $(P)$ describes a flow in porous media. Indeed if we denote by $u=u(x,t)$ the density of a fluid and define the velocity of the fluid as
\begin{equation}\label{fluids}
\vec{U} = -\nabla\Phi - \nabla u,
\end{equation}
where $\nabla\Phi$ is the external velocity field given by $\Phi$,  then \eqref{fluids} and the incompressibility condition 
\begin{equation}\label{incompressible}
\nabla\cdot\vec{U}=0
\end{equation}
yields $(P)$.

 When $\Phi=0$ and there is a  fixed boundary in the positive phase through which the fluid is injected, \eqref{fluids} and \eqref{incompressible}  yield the classical one-phase Hele-Shaw problem \cite{hs}.
 In this article however, our goal is to derive $(P)$ from a model problem in crowd motion with hard congestion, as described below.

\subsection{ A model in congested crowd motion}
 
Let us recall the transport problem with density constraint, introduced in \cite{mrs2}-\cite{mrs1}. Formally the problem can be written as the following: we look for a solution $\rho:\R^d\times [0,\infty) \to [0,\infty)$ satisfying
\begin{equation}\label{transport}
\rho_t +\nabla\cdot (\rho \nabla\Phi) = 0  \hbox { if } \rho<1, \hbox{ and } \rho \leq 1\hbox{ for all times.}
\end{equation}

The density constraint is natural in many settings, and it describes motion of congested individuals.  We refer to the articles \cite{mrs2, mrs1, s} for applications and mathematical formulations of the problem \eqref{transport}.  More rigorously, the problem can be written as 
\begin{equation}\label{weak_eqn}
\rho_t + \nabla\cdot(\rho\mathbf{u})=0, \quad \mathbf{u} = P_{C_\rho} \nabla\Phi,
\end{equation}
where $P$ is the projection operator and $C_\rho$ is the space of $L^2$-admissible velocity fields which do not increase $\rho$ on the saturated zone $\{\rho=1\}$. We refer to \cite{mrs1} for further description of $C_{\rho}$. Due to the low regularity of the velocity field $\mathbf{u}$ and the non-continuous dependence of the operator $P_{C_{\rho}}$ with respect to $\rho$, classical methods to study transport equations do not apply to \eqref{weak_eqn}. Indeed the uniqueness of solutions for \eqref{weak_eqn} is an open question, and is probably false without further conditions on the solution, given its hyperbolic nature.

In \cite{mrs1}, the authors study the connection between the PDE \eqref{weak_eqn} with $\rho_\infty$, which is the gradient flow of the following functional $E_\infty$ with respect to the 2-Wasserstein distance:
\begin{equation}\label{original}
E_{\infty}[\rho]:=\begin{cases}
\int_{\mathbb{R}^d} \rho(x)\Phi(x) dx & \text{ for }\|\rho\|_\infty \leq 1\\
+\infty & \text{ for } \|\rho\|_\infty > 1.
\end{cases}
\end{equation}
Further, they prove that when $\Phi$ is $\lambda$-convex, the gradient flow solution $\rho_\infty$ is a weak solution for \eqref{weak_eqn}. However, the full characterization of the solution and further qualitative properties of the solution remain open due to the lack of available methods to study \eqref{weak_eqn}.  The connection between $\rho_{\infty}$ and $(P)$ has been hinted, but only formally in the context of particle velocity.

\medskip

\begin{figure}
\begin{tikzpicture}\Large
  \matrix (m) [matrix of math nodes,row sep=3em,column sep=12em,minimum width=3em] {
  	 &\mbox{HS-D}\\
	\mbox{PME-D}   &\\
   &\rho_\infty  \\};
  \path[-stealth]
	(m-1-2) edge [double distance = 2pt, <->] node[right] {equal a.e.} (m-3-2)     
	    (m-2-1) edge [thick, ->, >=angle 90]
	    	node [sloped,above] {\small{locally uniformly as $m \to \infty$}}
		node [sloped, below] {\small{Theorem 1.1 (a)}}
	    (m-1-2)
	   
	    (m-2-1) edge [thick, ->, >=angle 90]
	    	node [sloped, above] {\small{ in  $W_2$ distance as $m \to \infty$}} 
	   	node [sloped, below] {\small{Theorem 1.1 (b)}}
	    (m-3-2)
	    ;
\node at (4.3,-0.5)  {\small{Theorem 1.2}};
\end{tikzpicture}
\caption{\label{fig:commutDiag}This diagram is a summary of the results of Theorems~\ref{main} and ~\ref{main00}.  Here $\rho_\infty$ denotes the gradient flow solution in the continuum limit, which in particular is a solution of \eqref{weak_eqn}. }
\end{figure}
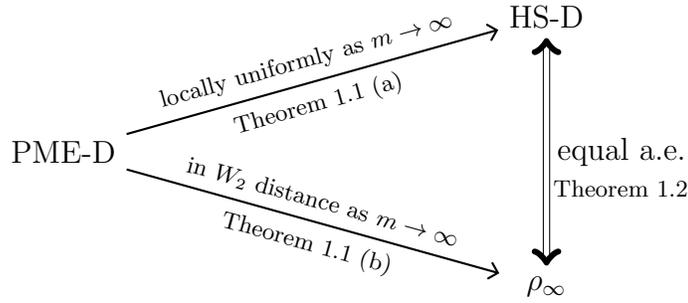

$\circ$ {\it Our contribution: }  In this article, our main focus is on establishing the connection between the free boundary problem $(P)$ and the gradient flow of $E_\infty$ in the setting of {\it patches}, i.e. when the initial data  is given as a characteristic function of a compact set $\Omega_0$, which we denote by $\chi_{\Omega_0}$. Note that since $\Phi$ is assumed to have a positive Laplacian, solutions tend to aggregate and thus we expect that the gradient flow $\rho_\infty(\cdot,t)$ will stay as a characteristic function at all times $t>0$. 

We show that the preservation of patches is indeed the case, and moreover the gradient flow solution $\rho_\infty(\cdot,t)$ indeed coincides with the characteristic function of the set $\Omega_t$, which evolves according to our problem $(P)$ with the initial support $\Omega_0$ (see Theorem~\ref{thm:coincide} below). This result enables us to characterize the evolution of $\rho_{\infty}$ in a unique way and also helps to understand the geometric behavior of $\rho_{\infty}$.  A summary of our results is shown in Figure~\ref{fig:commutDiag}.

In our analysis, the main challenge is the low regularity of $\rho_{\infty}$, since a priori we only know that it is in $C_W([0,T], \mathcal{P}_2(\mathbb{R}^d))$ (see Theorem \ref{ags}(b)). Thus it is rather difficult to directly study the geometric property of $\rho_{\infty}$. Instead of trying to directly show the link between the free boundary problem $(P)$ with the gradient flow $\rho_\infty$,  we use an approximation with degenerate diffusion. It has been formally suggested in \cite{mrs2} and \cite{s} that one could consider approximating the gradient flow of $E_\infty$ by the unconstrained gradient flow problem with the energy
\begin{equation}\label{approx}
E_m[\rho]:= \int \left(\frac{1}{m} \rho^m + \rho\Phi\right) dx.
\end{equation}
 It is well known  (for example, see \cite{o}) that the gradient flow $\rho_m$ associated with $E_m$ solves the porous medium equation with drift
 \begin{equation}\label{pme_drift}
\rho_t  -\nabla \cdot (\nabla (\rho^m)+ \rho\nabla \Phi)=0.
\end{equation}

Let us denote $\rho_m$ as the viscosity solution to \eqref{pme_drift} with initial data $\chi_{\Omega_0}$. We will prove that as $m\to\infty$, $\rho_m$ on the one hand converges to $\chi_{\Omega_t}$ locally uniformly, and on the other hand  converges to $\rho_\infty(\cdot, t)$ in 2-Wasserstein distance. Thus it follows that $\chi_{\Omega_t}$ and $\rho_\infty$ must be equal to each other almost everywhere. The main ingredients of the proof consist of stability results from viscosity solution theory and optimal transport theory, both of which rely strongly on the convexity-type conditions on $\Phi$. We also obtain comparison results and qualitative rates of convergences; see section 1.2 for precise statements. \subsection{Summary of results} 
 
 We are now ready to state our main results. The relevant assumptions, besides $(\textbf{A1})$ in the introduction, are stated in the beginning of section 4. 
 \begin{theorem} \label{main}
Let $\Omega_0$ be a compact set in $\R^d$ with locally Lipschitz boundary, and consider the initial data $u_0$ as given in \eqref{initial}. Then the following holds:

\begin{itemize}
\item[(a)] (Theorem~\ref{convergence}) Assuming \textup{(\textbf{A1})}, there exists a unique family of compact sets $\Omega_t$ in $\R^d$ starting from $\Omega_0$ such that any viscosity solution $u$ of $(P)$ satisfies $\overline{\{u(\cdot,t)>0\}} = \Omega_t$ for all $t>0$. Furthermore, let $\rho_m$ denote the viscosity solution to \eqref{pme_drift} with initial data $\chi_{\Omega_0}$. Then as $m\to \infty$, $\rho_m$ converges to $\bar{\rho}:=\chi_{\Omega_t}$ locally uniformly in $\R^d-\partial\Omega_t$ at each time $t>0$.\\
\item[(b)] (Theorem~\ref{main2}) Assume \textup{(\textbf{A2})} and \textup{(\textbf{A3'})}, and consider $\rho_0\in\mathcal{P}_2(\R^d)$ with $\|\rho_0\|_\infty \leq 1$ and $\int \rho_0 \Phi dx \leq M$. Let $\rho_m(x,t)$ denote the viscosity solution of \eqref{pme_drift} with initial data $\rho_0$. Then there exists $\rho_{\infty}\in C_W([0,T];\mathcal{P}_2(\R^d))$ such that  for any $T>0$, as $m\to\infty$, $\rho_{m}(\cdot,t)$ converges to $\rho_{\infty}(\cdot,t)$ in 2-Wasserstein distance, uniformly in $t$ for   $t\in [0,T]$, with the following convergence rate:
$$\sup_{t\in[0,T]}W_2(\rho_m(t), \rho_\infty(t)) \leq \frac{C(M,T,\|\Delta \Phi\|_\infty)}{m^{1/24}}.$$ 
\end{itemize}
\label{main_theorem}
\end{theorem}

Combining Theorem~\ref{main} (a) and (b), we immediately draw the following conclusion for the identification of $\rho_{\infty}$.

\begin{theorem} \label{main00} [Characterization of $\rho_{\infty}$]
Let $\Omega_0$, $\rho_{\infty}$ and $\bar{\rho}$ as given in Theorem \ref{main_theorem}. If \textup{(\textbf{A1})}, \textup{(\textbf{A2})} and \textup{(\textbf{A3'})} hold and if $\rho_0 = \chi_{\Omega_0}$, then $\rho_{\infty} = \bar{\rho}$ a.e.
\label{thm:coincide}
\end{theorem}

As a by-product of our analysis, we also show that a version of comparison principle holds between solutions to the discrete Jordan-Kinderlehrer-Otto (JKO) steepest descent scheme:

\begin{theorem} [Comparison principle, see Theorem \ref{thm:comparison}]  Let $\Phi$ satisfy \textup{(\textbf{A3})}.  For $2<m\leq \infty$, consider the two densities $\rho_{01}\in \mathcal{P}_{2,M_1}(\mathbb{R}^d)$, $\rho_{02} \in \mathcal{P}_{2,M_2}(\mathbb{R}^d)$ ($\mathcal{P}_{2,M_i}$ is as defined in section \ref{subsec:comparison}) with the property $M_1\leq M_2$ and $\rho_{01} \leq \rho_{02}$ a.e. (In the case $m=\infty$, we require in addition that $\|\rho_{0i}\|_\infty \leq 1$ for $i=1,2$).   For  given $h>0$, let $\rho_1, \rho_2$ be the respective minimizers of the following schemes:
\begin{equation}
\rho_i := \underset{\rho \in \mathcal{P}_{2,M_i}(\mathbb{R}^d)} {\argmin}\left[  E_m[\rho] + \frac{1}{2h}W_2^2(\rho, \rho_{0i})\right] \quad\text{ for }i=1,2,
\label{def:rho_i}
\end{equation}
Then $\rho_1 \leq \rho_2$ a.e..
\end{theorem}

This comparison result is  new in the context of Wasserstein distances and might be of independent interest (see section \ref{subsec:comparison} for more discussions). As a consequence one obtains geometric properties of the discrete solutions such as the confinement property (Corollary \ref{confinement}). 

Lastly, making use of this confinement result, for strictly convex $\Phi$ (but not necessarily uniformly convex), we have the following result concerning the long time behavior of $\rho_{\infty}$ starting from general initial data:
\begin{theorem}[Convergence to the stationary solution, see Theorem~\ref{long_time_convergence}] 
Let $2<m\leq \infty$.  Let $\Phi$ be strictly convex and satisfy \textup{(\textbf{A2})} and \textup{(\textbf{A3')}}.  Assume the initial data $\rho_0 \in \mathcal{P}_2(\mathbb{R}^d)$ has compact support, and in addition satisfies $\|\rho_0\|_\infty \leq 1$ in the case $m=\infty$. For $2<m\leq \infty$, let $\rho_{m}$ be given as the gradient flow for $E_m$ with initial data $\rho_0$, as defined in Theorem \ref{thm:collection}(b). Then as $t\to\infty$, $\rho_{m}(\cdot,t)$ converges to the unique global minimizer $\rho_S$ of $E_m$ exponentially fast in 2-Wasserstein distance.
\end{theorem}

\subsection{An outline of the paper}
 
  In section 2  we introduce the notion of viscosity solutions for $(P)$ and state basic properties of solutions. This part is  largely parallel to \cite{k}. In section 3 we show Theorem~\ref{convergence}. A key ingredient in this section is Theorem~\ref{pmelimit}, which identifies properties of the half-relaxed  limits of $\rho_m$ as $m\to\infty$. We point out that such convergence is previously known without the presence of the drift (\cite{gq},\cite{k}), but the presence of the drift and the resulting non-monotonicity of the support $\{\rho(\cdot,t)>0\}$ causes new challenges. In particular the weak formulation used in \cite{gq} based on variational inequalities no longer applies, and thus we proceed with the viscosity solutions approach similar to those taken in \cite{k}.  The argument presented in Theorem~\ref{pmelimit} is  of independent interest: it presents a strong stability argument which would apply to a general class of non-monotone free boundary problems.  Let us point out that the assumption $(\textbf{A1})$ not only justifies $(P)$ but also ensures the non-generacy of solutions of $(P)$ near the free boundary which leads to stability properties (see e.g. the proof of Theorem~\ref{pmelimit}.)

In section 4  we introduce the corresponding discrete-time schemes with free energy $E_m$ and $E_\infty$ respectively, and we study the convergence of the discrete solutions (and continuous gradient flow solutions) as $m\to\infty$.  There are new difficulties in handling the singular limit $m\to\infty$, since the discrete solutions $\rho_{m,h}$  corresponding to free energy \eqref{approx} are not necessarily less than $1$.  Lemma ~\ref{lemma:estimate_m_2} ensures that $\rho_{m,h}$ can be approximated with a density less than $1$  which is close to the original solution in $W_2$ distance and has similar energy $E_m$.  This approximation as well as estimates between $\rho_{m,h}$ and $\rho_{\infty,h}$ obtained in Proposition~\ref{prop:one_step_m_infty} enable us to prove Theorem~\ref{main2}.  Finally,  by combining the uniform convergence results obtained in Theorem~\ref{convergence} and Theorem~\ref{main2}, we conclude with Theorem~\ref{thm:coincide}. Let us mention that the $\Gamma$-convergence approach (\cite{dm} - \cite{ser}) may apply here to derive the convergence of $\rho_m$ to $\rho_{\infty}$ in 2-Wasserstein distance. On the other hand our approach is more quantitative and thus provides, for example, convergence rates in terms of $m$.

Finally, in section 5, for any fixed $2<m\leq \infty$, we present a comparison principle between solutions $\rho_{m,h}$ of the discrete-time scheme corresponding to free energy $E_m$ when $\Phi$ is semi-convex (Theorem~\ref{thm:comparison}). As mentioned above this result  is new for the (discrete) gradient flow solutions in the setting of Wasserstein distances.  As applications of the comparison principle, we discuss some confinement results and the long time behavior of $\rho_{m}$ for convex $\Phi$ in section 5.2 (Theorem~\ref{long_time_convergence}).

\subsection{Remarks on possible extensions}

 For simplicity of the presentation we did not consider the most general setting our approach could handle.  Below we discuss several situations where our approach (partially) extends.
\begin{itemize} 
\item[1.] Our approach would apply, with little modification, to the problem confined in a domain $\Sigma\subset \R^d$ with Neumann boundary data. On the other hand our approach would not apply, at least in its direct form, if one puts an exit (e.g. Dirichlet) condition on parts of $\partial\Sigma$. The challenge is in showing the convergence of discrete solutions, due to the fact that the $\lambda$-convexity of the associated energy no longer holds. On the other hand, the analysis in section 2 and 3 should still go through to yield that the solution of \eqref{pme_drift} converges to the solution of $(P)$ in domain $\Sigma$, with corresponding boundary conditions. \\

\item[2.]  In the case that $\Delta \Phi$ is not necessarily positive, and for general initial data $0\leq \rho_0 \leq 1$, the results in sections 4 and 5 are still valid and one can conclude that the solutions $\rho_m$ of \eqref{pme_drift} uniformly converges to a limiting profile $\rho_{\infty}$ in $2$-Wasserstein distance. In this case, the jammed region $\{\rho_{\infty}(\cdot,t)=1\}$ no longer satisfies finite speed of propagation and may nucleate at times. Due to this reason further characterization of $\rho_{\infty}$ beyond as a weak solution of \eqref{weak_eqn} remains open.\\

\item[3.] As mentioned above, without ($\textbf{A1}$), the unique characterization of the continuum limit $\rho_\infty$ given as the limit of discrete-time solutions  remains open.   In this case we suspect that a rather unstable {\it mushy region} $\{0<\rho_{\infty} < 1\}$ would develop in the limit $m\to\infty$, generating non-uniqueness of $\rho_{\infty}$. A similar difficulty arises in the analysis of  \cite{perthame} where singular limits of degenerate reaction diffusion equations are considered.

\end{itemize}

\section{ On the continuum solutions}

In section 2 and 3 we assume that $\Phi$ satisfies $(\textbf{A1})$. As mentioned before,  we do not expect classical solutions to exist globally either for $(P)$ or \eqref{pme_drift}. Hence to investigate qualitative behavior of solutions we begin by introducing the notion of weak solutions for $(P)$, in our case  the {\it viscosity solutions}. This notion of solutions is particularly useful when we are interested in the stability properties of interface problems. Let us point out that solutions of $(P)$ may be discontinuous due to the quasi-static nature of the evolution, and due to the singularity of the free boundary. Therefore in the definition of viscosity solutions we need to consider semi-continuous functions, in contrast to the definitions of viscosity solutions in section 3. We introduce a definition using comparison with smooth functions similar to the one in \cite{k} and \cite{Pozar}.

\begin{definition}\label{subsolution}
A nonnegative upper-semicontinuous function u defined in $Q := \R^d \times [0,\infty)$ is a viscosity subsolution of $(P)$ with compactly supported initial data $u_0$ if the following hold:
\begin{itemize}
\item[(a)] $u = u_0$ at $t = 0$ and  $\{u_0>0\} = \overline{\{u(x,t) > 0\}}\cap\{t=0\}$;
\item[(b)] $\{u>0\} \cap\{t\leq \tau\} \subset \overline{\{u>0\}\cap\{t<\tau\}}$ for every $\tau>0$ ;
\item[(c)] For every $\phi \in C^{2,1}(Q)$ that has a local maximum of $u - \phi$ in $\overline{\{u>0\}} \cap \{t\leq t_0\}$ at $(x_0,t_0)$,\\
	
\begin{itemize}
	\item[(i)]  if $(x_0, t_0) \in \{u > 0 \}$, $-\Delta \phi(x_0,t_0) \le \Delta \Phi(x_0)$.
	\item[(ii)] if $(x_0, t_0) \in \partial \{u>0\}$, $u(x_0, t_0) = 0$, and if $|\nabla \phi(x_0,t_0)| \ne 0$, then
\[\min(-\Delta \phi - \Delta \Phi, \phi_t - |\nabla\phi|^2-\nabla \phi \cdot \nabla \Phi)(x_0, t_0) \le 0.\] 
	\end{itemize}

\end{itemize}
\end{definition}

Note that the condition (c)(ii) is to ensure that limits of viscosity solutions are viscosity solutions, since the boundary can collapse in a limit and then boundary points of the limiting functions becomes interior points of the limit.

\begin{definition}
A nonnegative lower-semicontinuous function $v$ defined in $Q := R^d \times [0,\infty)$ is a viscosity supersolution of $(P)$ with initial data $v_0$ if the following hold:
\begin{itemize}
\item[(a)] $v = v_0$ at $t = 0$.
\item[(b)] For every $\phi \in C^{2,1}(Q)$ that has a local minimum zero of $v - \phi$ in $\R^d\times (0,t_0]$  at $(x_0, t_0)$,\\

	\begin{itemize}
	\item[(i)]  if $(x_0, t_0) \in \{v > 0\}$, $-\Delta \phi(x_0,t_0) \ge \Delta \Phi(x_0)$.
	\item[(ii)] if $(x_0, t_0) \in \partial \{v>0\}, v(x_0, t_0) = 0$ and if 
\begin{align}\label{bdryTouchingCond}
  |\nabla \phi(x_0, t_0)| \ne 0 \mbox{ and } \{\phi > 0\} \cap \{v >0 \} \cap B(x_0,t_0) \ne \emptyset\mbox{ for some ball } B
\end{align} then
\[\max(-\Delta \phi - \Delta \Phi, \phi_t - |\nabla\phi|^2-\nabla \phi \cdot \nabla \Phi)(x_0, t_0) \ge 0. \]
	\end{itemize}
\end{itemize}
\end{definition}
The condition \eqref{bdryTouchingCond} is to ensure that $\phi$ touches $v$ from below in a non-degenerate way.

Let us define, for a function $h$ in $Q$, the upper and lower semi-continuous envelopes of $h$:
\begin{equation}\label{envelopes}
h^*(x,t): = \lim_{\e\to 0} \sup_{\{|x-y|, |t-s| \leq \e\}} h(y,s), \quad h_*(x,t):= \lim_{\e\to 0} \inf_{\{|x-y|, |t-s| \leq \e\}} h(y,s).
\end{equation}

\begin{definition}
$u$ is a viscosity solution of $(P)$ with initial data $u_0$ if $u_*$ and $u^*$ are respectively viscosity sub- and supersolutions of $(P)$ with initial data $u_0$.
\end{definition}

We will discuss several properties of viscosity solutions which will be used in the main theorem of the article. 

\subsection{Inf- and Sup-convolutions}

Next we introduce regularizations for viscosity solutions of $(P)$, which is by now standard  for free boundary problems (see e.g. \cite{cv}). Given a viscosity subsolution $u$ and $r > 0$, we define
\begin{equation}\label{sup_convolution}
\overline{u}_r = \sup_{B_r(x,t)} u(y,\tau) \quad\hbox{ for } t\geq r
\end{equation}
and likewise given a viscosity supersolution $v$, and $r, \delta > 0$ with $\delta << r$, we define
\begin{equation}\label{inf_convolution}
\underline{v}_r = \inf_{B_{r - \delta t}(x,t)} v(y,\tau) \quad\hbox{ for } t\geq r.
\end{equation}
These are called the sup- and inf- convolutions, respectively, and serve to smooth out viscosity solutions to help analyze the speed of the free boundary.
The following properties of $\overline{u}_r$ and $\underline{v}_r$ are direct consequences of their definitions.
\begin{lemma}\label{convolutions}
\begin{itemize}
\item[(a)]$\overline{u}_r$ is a viscosity subsolution of $(P)$.  Moreover, at each point $(x_0,t_0)\in\partial\{\overline{u}_r>0\}$ there exists a space-time {\it interior ball} $B$ such that 
$$
B\subset\{\bar{u}_r>0\}\hbox{ and }  \overline{B} \cap \{\bar{u}_r=0\} = \{(x_0,t_0)\}. 
$$
\item[(b)]$\underline{v}_r$ is a viscosity supersolution of $(P)$.  Moreover, at each point $(x_0,t_0)\in\partial\{\underline{v}_r>0\}$ there exists a space-time {\it exterior ball} $B$ such that 
$$
B\subset\{\underline{v}_r=0\}\hbox{ and }  \overline{B} \cap \overline{\{\underline{v}_r>0\} }= \{(x_0,t_0)\}. 
$$
\end{itemize}
\end{lemma}

Let $e_{n+1}$ denote the vector $(0,...,1)$ in $Q$. The following two lemmas will prove useful in our analysis later. The first lemma can be proven with a parallel proof to that of Lemma 2.5 in \cite{k} and thus we omit the proof.  The second lemma is more interesting and involves ruling out the case of local total collapse of the solution, that is, the solution completely vanishing at a given time.  The proof relies on $(\textbf{A1})$ to build a quadratic barrier subsolution.

\begin{lemma}[$\{\bar{u}_r>0\}$ cannot expand with infinite speed] \label{u_r_fact}
Suppose $(x_0,t_0)\in\partial\{\overline{u}_r > 0\}$. Then the corresponding interior ball cannot have its outward normal as $e_{n+1}$ at $(x_0,t_0)$.
\end{lemma}

\begin{lemma} [$\{\underline{v}_r>0\}$ cannot shrink with infinite speed]\label{v_r_fact}
Suppose $(x_0,t_0)\in\partial\{\underline{v}_r > 0\}$. Then the corresponding exterior ball cannot have its outward normal as $-e_{n+1}$ at $(x_0,t_0)$. 
\end{lemma}

\begin{proof} 
1. Suppose that $\{\underline{v}_r>0\}$ has an exterior ball with outward normal $-e_{n+1}$ at a point $(x_0,t_0)$.  Then at $(x_0,t_1)$, $v$ will have an interior ball $B_1$ centered at $(x_0,t_0)$ where $t_1 - t_0 = r - \delta t_0$ and $B_1$ has outward normal $e_{n+1}$ at $(x_0,t_1)$.
\newline\newline
2.  Fix a number $\lambda$ satisfying 
\[\lambda < \frac{1}{5 \max_{B_2(x_0,t_0)} |\nabla \Phi|},\quad \lambda< 1, \quad \lambda << r - \delta t_0.\] 
3.  We define
\[\omega(x,t) := v\left(\lambda x +x_0, \lambda^2 (t - 1)+t_1\right)\]
This serves to map the cylinder
\[C_0 := \{(x,t) : |x - x_1| < \lambda, t_1 - \lambda^2 < t < t_1\}\]
 to the cylinder $C := \{|x| < 1\}\times [0,1]$.  Then $\omega$ is a viscosity solution of a re-scaled version of $(P)$:
\[
\left\{\begin{array}{lll}
  \Delta \omega + \lambda^2 \Delta \Phi_1= 0 & \hbox{ in }& \{\omega > 0\};\\
   V= -\partial_{\nu}\omega -\lambda\partial_{\nu}\Phi_1 & \hbox{ in }& \partial\{\omega > 0\},
\end{array}\right.
\]
where $\Phi_1$ is a rescaled and recentered version of $\Phi$.  By our choice of $\lambda$, the bottom of $C$ is strictly contained in $B_1$, and so by lower semi-continuity we can find $\epsilon > 0$ satisfying $\omega > \epsilon$ at $t = 0$.  
\newline\newline
4.  We construct our barrier.  Define
\[\varphi:= \alpha(1 - t / 5 - x^2/2)\]
where we choose $\alpha>0$ so that $ \alpha < \min(\inf_{C} \Delta \Phi_1, \epsilon)$.  Then $-\Delta \varphi = \alpha < \Delta \Phi_1$, and on the bottom of $C$, $\varphi < \epsilon < \omega$.  On the sides of $C$, $\varphi < 0 < \omega$, so $\varphi < \omega$ on the parabolic boundary of $C$. However, $\varphi(0,1) = 4\alpha/5 > 0 = \omega(0,1)$, so they eventually cross.
\newline\newline
5.  We examine their crossing.  To this end, we define $T$ to be the first crossing time of $\varphi$ and $\omega$:
\[ T := \inf \{t \ge 0 | \mbox{ there exists } x \in C \mbox { s.t. } \omega(x,t) -\varphi(x,t) < 0\}\]
Then we can find a sequence $(x_n, t_n)$ with $t_n \downarrow T$ and 
\[\omega(x_n,t_n) - \varphi(x_n,t_n) \le 0\]
  Now we are in a compact set so we can suppose that $x_n \to \bar{x} \in C$, and since $\omega$ is lower semi-continuous, we must have that 
\[\omega(\bar{x},T) - \varphi(\bar{x},T) = -\beta \le 0\]
Then $(\bar{x},T)$ must be in the parabolic interior of $C$.  

The fact that this is a local minimum of $\omega - \phi$ follows since it is the first time $\omega$ and $\phi$ cross.  We are done now because $-\Delta \varphi(\bar{x},T) = \alpha < \Delta \Phi_1(\bar{x},T)$ and
\begin{align*}
  \varphi_t - |\nabla \varphi|^2 - \lambda \nabla \varphi \cdot \nabla \Phi_1 &= -\alpha / 5-\alpha^2 x^2 - \lambda \alpha x \cdot \nabla \Phi_1 < 0
\end{align*}
where the final inequality comes from our assumption on $\lambda$.
\end{proof}

\subsection{Comparison principle} 

The central property of the viscosity solution theory is in the comparison principle, which we state below. The proof is mostly parallel to that of \cite{k}, and thus we only sketch the outline of the proof.

We say two functions $u,v:\R^d\to \R^+$ are {\it strictly separated}, denoted by $u\prec v$,  if 
$$
u <v\hbox{ in }\overline{\{u>0\}}\quad\hbox{ and } \overline{\{u>0\}}\hbox{ is a compact subset of } \{v>0\}.
$$

\begin{theorem}\label{thm:cp}
Let $u$ and $v$ be respectively viscosity sub- and supersolutions of $(P)$. If $u(\cdot,0) \prec v(\cdot,0)$ then $u(\cdot,t) \leq v(\cdot,t)$ for all $t>0$.
\end{theorem}

\textbf{Sketch of the proof}

1. Due to the fact that $u_0 \prec v_0$, applying Definition 2.1 (a)-(b) and the semi-continuities of $u$ and $v$, we have $\bar{u}_r(\cdot,r)\prec \underline{v}_r(\cdot,r)$ for sufficiently small $r>0$.

\medskip

2. We claim that $\bar{u}_r \leq \underline{v}_r$ for all times bigger than $r$, which yields our theorem. Hence suppose not, and define
$$
t_0:= \sup\{t:u_r(\cdot,s) \prec v_r(\cdot,s) \hbox{ for } s\leq t \} <\infty.
$$
One can then proceed as in \cite{k}, using the above lemmas to exclude the possibility that $\bar{u}_r$ and $\underline{v}_r$ cross over each other discontinuously in time,  to show that at $t=t_0$, there is a point $x_0$ such that 
$$
(x_0,t_0)\in\partial\{\bar{u}_r>0\}\cap\partial\{\underline{v}_r>0\}.
$$
Moreover,  there exists an interior ball $B_1$ to $\{\bar{u}_r>0\}$ and an exterior ball $B_2$ to $\{\underline{v}_r>0\}$ at $(x_0,t_0)$ such that
$$
\overline{B_1}\cap \overline{B_2}\cap\{t\leq t_0\} = (x_0,t_0).
$$

Let $(\nu,m_1)$ be the interior normal to the interior ball $B_1$ and $(\nu,m_2)$ be the exterior normal to the exterior ball $B_2$ at $(x_0,t_0)$, with $|\nu| = 1$. Due to the Lemmas ~\ref{u_r_fact} and \ref{v_r_fact},  $m_1$ and $m_2$ are both finite. In particular at $t=t_0$ both the sets $\{\bar{u}_r>0\}$ and $\{\underline{v}_r>0\}$ have the interior space ball $B_1\cap\{t=t_0\}$ with interior normal $\nu$.   Since $\bar{u}_r$ crosses $\underline{v}_r$ from below at $(x_0,t_0)$, we have $m_1 \geq m_2$. Moreover at $(x_0,t_0)$, the support of $\bar{u}_r$ propagates faster than normal velocity $m_1$, and the support of $\underline{v}_r$ slower than normal velocity $m_2$. Formally speaking, we would like to claim from the fact that $\bar{u}_r$ and $\underline{v}_r$ are respectively sub- and (strict) supersolutions of $(P)$ that 
\begin{equation}\label{claim0}
|\nabla\bar{u}_r| \geq m_1+ \mu \hbox{ and } |\nabla\underline{v}_r| < m_2+\mu, \hbox{ where } \mu= \nabla\Phi(x_0,t_0) \cdot \nu.
\end{equation}
From the claim, we deduce a contradiction since $\bar{u}_r\leq \underline{v}_r$ at $t=t_0$ and $m_1 \geq m_2$.

\medskip

3.  To prove \eqref{claim0} in the viscosity sense, we can use appropriate barriers to compare with $\bar{u}_r$ and $\underline{v}_r$, to measure the growth of these functions at $x_0$. This part of the proof is parallel to that of Theorem 2.2 in \cite{k}. Indeed the barriers corresponding to our problem $(P)$ are constant multiples of the ones constructed in Appendix A of \cite{k}. 
\hfill$\Box$

\begin{remark}\textup{
Let us point out that, due to the restriction on the strict separation of the initial data, the above comparison principle does not immediately yield the uniqueness of the solutions for $(P)$. Later in the paper we will derive the uniqueness result (see Theorem~\ref{convergence}),  by showing that $L^1$-contraction holds between the characteristic functions of  the positive sets of the viscosity solutions. }
\end{remark}

\section{Approximation by degenerate diffusion with drift}

As in section 2 we continue to assume $(\textbf{A1})$. Let $\rho$ be a weak, continuous solution of \eqref{pme_drift}, as given in \cite{v}. We define the pressure variable $u$ by 
\begin{equation}\label{pressure}
u:= \frac{m}{m-1} \rho^{m-1}.
\end{equation}
 Then $u$ formally solves 
$$
  u_t = (m-1)u(\Delta u + \Delta \Phi) + |\nabla u|^2 + \nabla u \cdot \nabla \Phi.
  \leqno\mbox{(PME-D)}_m
$$

In \cite{cv} (for $\Phi=0$) and in \cite{kl} it was shown that $u$ is a viscosity solution of (PME-D)$_m$.
For completeness we review the definitions.  First we define a {\it classical solution} of  (PME-D)$_m$  as a nonnegative function $u \in C^{2,1}(\overline{\{u > 0\}})$ that 
\begin{itemize}
\item[(a)] solves (PME-D)$_m$ in $\{u > 0\}$,
\item[(b)] has a free boundary $\Gamma = \partial\{u > 0\}$ which is a $C^{2,1}$ hypersurface, and
\item[(c)] $\Gamma$ evolves with the outer normal velocity $ |\nabla u|+ \eta \cdot \nabla \Phi,$ where $\eta$ is the inward normal of $\Gamma$.
\end{itemize}

We then use the classical solutions as test functions to define viscosity solutions of (PME-D)$_m$.
\begin{definition}
A non-negative continuous function u defined in $Q := \R^d \times (0,\infty)$ is a viscosity subsolution of  (PME-D)$_m$ if for every $\phi \in C^{2,1}(Q)$ that has a local maximum zero of $u - \phi$ in $\{t \le t_0\}$ at $(x_0,t_0)$,
\[(\phi_t - (m-1)\phi(\Delta \phi + \Delta \Phi) - |\nabla \phi|^2 - \nabla \phi \cdot \nabla \Phi)(x_0,t_0) \le 0.\]
\end{definition}

\begin{definition}
A continuous function $v: Q\to \R_+$ is a viscosity supersolution of (PME-D)$_m$ if:
\begin{itemize}
\item[(a)] For every $\phi \in C^{2,1}(Q)$ that has a local minimum zero of $v - \phi$ in $\{v > 0\} \cap \{t \le t_0\}$ at $(x_0, t_0)$,
\[\big(\phi_t - (m-1)\phi(\Delta \phi + \Delta \Phi) - |\nabla \phi|^2 - \nabla \phi \cdot \nabla \Phi\big)(x_0,t_0) \ge 0.\]
\item[(b)] Any classical  solution of (PME-D)$_m$ that lies below $v$ at time $t_1 \ge 0$ cannot cross $v$ at a later time.
\end{itemize}
Finally, $u$ is a viscosity solution of  (PME-D)$_m$ with compactly supported initial data $u_0$ if it is both a viscosity subsolution and supersolution of (PME-D)$_m$ and both $u(\cdot,t)$ and $\{u(\cdot,t)>0\}$ uniformly converge to $u_0$ and $\{u_0>0\}$ as $t\to 0$, respectively, in uniform norm and in Hausdorff distance.
\end{definition}

Let us point out that the above definitions, based on comparison with classical solutions, are essentially in the same spirit as the definition of viscosity solutions of $(P)$ introduced in section 2. 
\subsection{Properties of $u_m$ at the free boundary} 

We remark that the definition of viscosity supersolutions of (PME-D)$_m$ only applies in $\{v > 0\}$ in order to make the viscosity solution notion be equivalent to the idea of weak solutions.  This has the consequence of needing extra effort to analyze the behavior at the free boundary, which is provided by the following lemma.  Its proof is analogous to Lemma 1.7 in \cite{k}, with the difference in the construction of barriers.  

\begin{lemma}\label{PMEDfact}
Let $v$ be a viscosity supersolution (subsolution) of (PME-D)$_m$, and suppose that $\phi$ is a smooth function where $v- \phi$ has a local minimum (maximum) zero in $\overline{\{v > 0\}}$ at $(x_0, t_0) \in \partial\{v > 0\}$ with $t_0>0$.  If $\phi$ satisfies \eqref{bdryTouchingCond} at $(x_0, t_0)$, then
\begin{equation}\label{inequality}
(\phi_t - |\nabla \phi|^2 - \nabla \phi \cdot \nabla \Phi)(x_0, t_0) \ge (\le)\quad  0.
\end{equation}
\end{lemma}

\begin{proof}

First note that the subsolution case by definition is trivial as discussed above, since $\phi(x_0,t_0) = 0$.  Thus we proceed to the supersolution case.  We may set $t_0 = 0$ after a translation.  

Let us fix constants $r, \delta > 0$ and prove the lemma for the inf-convolution of $v$,
\[W(x,t) = \inf_{B_{r -\delta t}(x,t)} v (y,\tau)\]
Then the lemma follows by taking $\delta\to 0$ and then $r\to 0$.

Now suppose that for a smooth $\phi$, $W- \phi$ has a local minimum in $\overline{\{W > 0\}}$ at $(x_0,0) \in \partial\{W > 0\}$, with $\phi$ satisfying \eqref{bdryTouchingCond}. By perturbing $\phi$ we may assume that the minimum is strict.  Let $H$ be the hyperplane tangent to $\{\phi>0\}$ at $(x_0,0)$, with $(\nu,\gamma)$ the inward normal to $H$ with $|\nu| = 1$. Note that $\gamma>-\infty$ from Corollary 2.16 in \cite{kl}.  Let $\alpha = |\nabla\phi|(x_0,0)= \phi_\nu(x_0,0)>0$.  Towards a contradiction, we assume that \eqref{inequality} fails, and so it follows that for some $\sigma>0$
\begin{equation}\label{gammaupperbound}
\gamma=V_\phi = \frac{\phi_t}{\phi_\nu}(x_0,0) < (\alpha-\sigma)+ \nu \cdot \nabla \Phi(x_0).
\end{equation}
Hence $\gamma$ is finite.
Moreover we have
\begin{equation}\label{order}
W(x,t)\geq \phi(x,t) \hbox{ in } B_\eta(x_0)\times [-\eta,0] \quad \hbox{ for } \eta <<1.
\end{equation}

Due to the regularity of $\phi$, there exists a space ball $B_0$ interior to the set $\{x: \phi(x,0)>0\}$ with $x_0\in\partial B_0$. We define $\gamma_1$ as follows:
\[\gamma_1 := 
\left\{
     \begin{array}{lr}
      \gamma + \sigma / 4  & \mbox{ if } \gamma \ge \nabla \Phi(x_0)\cdot \nu\\
      \alpha/2 +  \nabla \Phi(x_0)\cdot \nu  & \mbox{otherwise}
     \end{array}
   \right.
\]
Then we use the result of Lemma~\ref{appendix} to find a classical subsolution $S$ of $(PME-D)_m$ in a neighborhood $B_\eta(x_0) \times [-\eta, \eta]$ that firstly has initial support inside $B_0$, secondly has advancing speed $\gamma_1$ at $(x_0,0)$, and lastly has a parameter $0<\epsilon<\min(\sigma/4, \alpha/4)$ such that $S$ satisfies 
\begin{equation}\label{almost}
\gamma_1 \geq  |\nabla S| + \nu\cdot \nabla \Phi  - \e \hbox{ at } (x_0,0).
\end{equation}
This condition  helps us to show that it initially lies under $\phi$.

 We now claim that $S$ lies under $W$ in $B_{\eta}(x_0) \times [-\eta,0]$ for sufficiently small $\eta$ , which will yield the desired contradiction to the fact that $S$ is a subsolution and $W$ is a supersolution, since $S$ will cross $W$ at $(x_0,0)$.

Due to \eqref{gammaupperbound} and \eqref{almost}, we have 
\begin{align}\label{below}
  |\nabla S| (x_0,0)&< 
\left\{
     \begin{array}{lr}
      \alpha - \frac{\sigma}{2}  & \mbox{ if } \gamma \ge \nabla \Phi(x_0)\cdot \nu\\
       3\alpha/4 & \mbox{ otherwise  }
     \end{array}
   \right. \\
& < \alpha = |\nabla \phi|(x_0,0).
\end{align}

On the other hand, observe that the support of $S$ propagates with the normal speed faster than that of $\phi$ at $(x_0,0)$ due to \eqref{almost}. Due to the regularity of $\phi$ and $S$ and their ordering at $t=0$ it then follows that 
\begin{equation}\label{support:order}
\overline{\{S>0\}}\subset \{\phi>0\} \hbox{ in } B_{\eta}(x_0)\times [-\eta,0]
\end{equation}
if $\eta$ is sufficiently small. From the above two inequalities it follows that $S \leq \phi$ in $B_{\eta}(x_0)\times [-\eta,0]$ if $\eta$ is sufficiently small. We can now conclude using the fact that $\phi\leq W$ in that neighborhood.

\end{proof}

\subsection{Characterization of the half relaxed limits of $u_m$ as $m\to\infty$}

\medskip

Let $\Omega_0$ and $u_0$ as given in the introduction, and let  $u_m$ be the unique viscosity solution to (PME-D)$_m$ with the initial data $u_0$. Recall that $u_m$ is given as the pressure variable of $\rho_m$ by \eqref{pressure}, where $\rho_m$ assumes the corresponding initial data $(\frac{m-1}{m}u_0)^{1/(m-1)}$. Let us then define
\begin{eqnarray*}
  u_1(x,t) &=& \inf_{n \ge 0} \sup_{\substack{ m \ge n \\|(x,t) - (y,s)| < 1/n }} u_m(y,s); \\  
  u_2(x,t) &=& \sup_{n \ge 0} \inf_{\substack{ m \ge n \\|(x,t) - (y,s)| < 1/n }} u_m(y,s). \\
\end{eqnarray*}
Note that the $\{u_m\}$ are uniformly bounded in $m$, as a consequence with comparison with stationary solutions of the form $(C-\Phi(x))_+$ with sufficiently large $C>0$. Hence $u_1$ and $u_2$ are both finite.
 
 Since we cannot guarantee that the support of $u_1$ traces those of $u_m$, we need to define an auxilliary function. Let us define the function 
$$
  \eta(x,t):=\limsup_{\substack{ m \to \infty \\(y,s) \to (x,t) }} \chi_{\{\supp(u_m)\}}(y,s),
$$
and the closure of the support of $\eta$:
$$
\Omega =\overline{\{(x,s): \eta(x,s)>0\}},\quad \Omega(t):= \Omega\cap\{s=t\}.
$$
Finally, let us define the largest subsolution of the Poisson equation $-\Delta w = \Delta\Phi$ supported in $\Omega$:
$$
\tilde{u}_1:= [\sup\{ v:\R^d\times(0,\infty)\to \R\hbox{ such that }  -\Delta v \leq \Delta \Phi \hbox{ and } v =0 \hbox{ outside of } \Omega\}]^*.
$$
Here $f^*$ denotes the upper semicontinuous envelope of $f$, as defined in \eqref{envelopes}. 

Note that then $\tilde{u}_1=0$ outside of $\Omega$ and for each $t>0$, $\tilde{u}_1$ satisfies
$$-\Delta\tilde{u}_1(\cdot,t) \leq \Delta\Phi \hbox{ in }\R^d, \quad -\Delta\tilde{u}_1(\cdot,t) = \Delta\Phi \hbox{ in the interior of } \Omega(t).
$$

 This auxiliary function $\tilde{u}_1$ is indeed the new component of the proof compared to the corresponding theorem in \cite{k}. We point out that $\tilde{u}_1$ is positive in the interior of $\Omega(t)$ due to $(\textbf{A1})$.

\begin{theorem}\label{pmelimit}
Let $u_1$, $u_2$ and $\tilde{u}_1$ be as given above, and let $\Phi$ satisfy $(\textbf{A1})$.
Then $\tilde{u}_1$ is a viscosity subsolution of $(P)$ in $Q$, and $u_2$ is a viscosity supersolution of $(P)$ in $Q$ with initial data $u_0$.
\end{theorem}

\begin{proof}  First note that $u_2$ is lower semicontinuous by its definition. Likewise, $u_1$ is upper semicontinuous.

\textit{A. $u_2$ is a supersolution: }

1. Suppose we have a smooth function $\phi$ and $u_2 - \phi$ has a local minimum at $(x_0, t_0)$ in \\$\overline{\{u_2 > 0\}}\cap \{t \le t_0\}$.  By adding $\epsilon(t-t_0) - \epsilon(x-x_0)^2 +c$ to $\phi$ one may assume that the minimum is zero, and is strict in $C_r\cap \overline{\{u_2 > 0\}}$, where $C_r := B_r(x_0) \times [t_0 - r, t_0]$ for  small $r>0$.  

If $(x_0, t_0)$ is in $\{u_2 > 0\}$, by lower-semicontinuity of $u_2$, we can make $r$ smaller and assume that $C_r \subset \{u_2 > 0\}$.  On the other hand if $(x_0, t_0) \in \partial \{u_2 > 0\}$, we can assume that \eqref{bdryTouchingCond} holds for $\phi$.  In particular, $|\nabla \phi| \ne 0$ so that $u_2 - \phi >0$ in $C_r \cap \{u_2 > 0\}^c$ away from $(x_0, t_0)$.  Thus in either case we can find that $u_2 - \phi$ has a strict local mininum zero in all of $C_r$.\\

2. We now claim the following:  if $r$ is sufficiently small, along a subsequence $u_m - \phi$ has a minimum at points $(x_m, t_m) \in C_r$ with $(x_m, t_m) \to (x_0, t_0)$ and $(x_m, t_m) \in \overline{\{u_m > 0 \}}$.  

To show this, define $(x_m, t_m) = \mbox{argmin}_{C_r} \;(u_m - \phi)$; we can assume that the sequence only ranges over $m$ that achieve the infimum of $u_2$ at $(x_0, t_0)$.  Let $(x',t')$ be a limit point of $\{(x_m, t_m)\}_m$.

First let us show that upon further refinement of our sequence we have $(x_m,t_m) \in \overline{\{u_m > 0 \}}$.   Clearly this is true if $(x_0,t_0)\in\{u_2>0\}$, and thus suppose $(x_0, t_0) \in \partial\{u_2 > 0\}$ and $(x_m,t_m)$ lies outside of the support of $u_m$.  Then we can assume that \eqref{bdryTouchingCond} holds for $\phi$, so in particular we can assume that $|\nabla \phi| \ne 0$ in $C_r$. This rules out the possibility that $(x_m,t_m)$ lies in the interior of $C_r$. Also, in this case $\phi(x_0, t_0) = 0$, and so we can find $\alpha > 0$ so that $\phi < -\alpha < 0$ on $\partial C_r \cap \{u_2 > 0 \}^c$.  This rules out the possibility that $(x_m,t_m)$ lies on the boundary of $C_r$. Thus we conclude that  $(x_m, t_m) \in \overline{\{u_m > 0\}}$ for sufficiently large $m$.

Next let us verify that $(x',t') = (x_0, t_0)$.  By definition for arbitrary $(y,s)$ in $C_r$ 
\begin{align}\label{thingweregonnatakelimof}
  (u_m-\phi)(y,s)  \ge  (u_m - \phi)(x_m,t_m).
\end{align}
Since $(x_m,t_m) \to (x',t')$, for each $n$ there is $M(n)$ so that $|(x_m,t_m) - (x',t')| < 1/n$ if $m \ge M$, and we may assume that $M(n) \ge n$.  Then
\[\inf_{m \ge M(n)} u_m(x_m,t_m) \ge  \inf_{\substack{ m \ge M(n) \\|(x',t') - (y,s)| < 1/n }} u_m(y,s) \ge \inf_{\substack{ m \ge n \\|(x',t') - (y,s)| < 1/n }} u_m(y,s).\]
Taking $\sup_n$ on both sides we find $\liminf_{m\to \infty} u_m(x_m,t_m) \ge u_2(x',t')$.
Then, taking $\liminf$ of both sides of ~\eqref{thingweregonnatakelimof} as $(y,s) \to (x_0, t_0)$ and $m \to \infty$, we find
\[(u_2-\phi)(x_0,t_0) \ge (u_2 - \phi)(x',t')\]
which contradicts that $(x_0,t_0)$ is the strict minimum of $u_2 - \phi$ in $C_r$.  This proves our claim.
\newline\newline
3. To finish showing that $u_2$ is a viscosity supersolution, take $\phi$ and $(x_m,t_m)$ as given above. When $(x_0, t_0) \in \{u_2 > 0\}$, a straightforward computation using the the properties of $u_m$ as viscosity solutions of (PME-D)$_m$ gives
\[-\Delta \phi(x_0,t_0) \ge  \Delta \Phi(x_0,t_0)\]
as needed. Next suppose $(x_0, t_0) \in \partial \{u_2 > 0\}$, and that \eqref{bdryTouchingCond}  holds for $\phi$.  Suppose towards a contradiction that there is $\alpha > 0$ so that
\begin{equation}\label{observation00}\max(-\Delta \phi - \Delta \Phi, \phi_t - |\nabla\phi|^2-\nabla \phi \cdot \nabla \Phi)(x_0, t_0) = -\alpha < 0. 
\end{equation}
Let us define $\phi_m := \phi + C(m)$ so that $(u_m - \phi_m)(x_m,t_m) = 0$.  
Since $(x_m,t_m)\to (x_0,t_0)$, \eqref{observation00} yields that 
\[\big((\phi_m)_t - (m-1)\phi_m(\Delta \phi_m + \Delta \Phi) - |\nabla \phi_m|^2 - \nabla \phi_m \cdot \nabla \Phi\big)(x_m,t_m) < 0,\]
which contradicts with the fact that $u_m$ is viscosity solution of (PME-D)$_m$. Thus we have $(x_m,t_m)\in\partial\{u_m>0\}$. But then the inequality 
\[
 ((\phi_m)_t - |\nabla \phi_m|^2 - \nabla \phi_m \cdot \nabla \Phi)(x_m,t_m) < -\alpha / 2<0.
\]
contradicts Lemma \ref{PMEDfact}, which applies since $\phi$ is smooth and so satisfies \eqref{bdryTouchingCond} at $(x_m, t_m)$ for large $m$.

\medskip

\textit{B. $\tilde{u}_1$ is a subsolution}

The subsolution part of our theorem is harder to prove, since a smooth test function touching $\tilde{u}_1$  at a free boundary point $(x_0,t_0)$ from above in $\overline{\Omega}$ cannot be extended smoothly to outside of $\Omega$ so that the order is preserved. Thus the proof of B requires a careful study of the behavior of the free boundary of $\tilde{u}_1$, which is achieved by studying the properties of $u_1$ and $u_m$. First note that Definition~\ref{subsolution} (b) is satisfied due to Theorem~\ref{pmeBddSpeed}. We proceed to show the property given in Definition~\ref{subsolution} (c). 

1.  It is straightforward from the definition of $\Omega(t)$ that 
 \begin{equation}\label{order0}
 \{u_1(\cdot,t)>0\} \subset \Omega(t).
 \end{equation}
Parallel arguments to the supersolution case yield that $-\Delta u_1 \le \Delta\Phi$ in $\{u_1 > 0\}$ in the viscosity sense. Thus   it follows that 
\begin{equation}\label{order}
u_1 \leq \tilde{u}_1.
\end{equation}
 Suppose that we have a smooth function $\phi$ and $\tilde{u}_1 - \phi$ has a strict local maximum at $(x_0, t_0)$ in $\Omega \cap \{t \le t_0\}$. As mentioned before  $\tilde{u}_1$ satisfies $-\Delta\tilde{u}_1(\cdot,t) \leq \Delta\Phi$, indeed with equality in the interior of $\Omega(t)$. Thus to check that $\tilde{u}_1$ is a subsolution, it is enough to consider the case when $x_0\in \partial\Omega(t_0)$ and  $\tilde{u}_1(x_0,t_0)=0$. Note that in this case ~\eqref{order} yields that $u_1-\phi$ also has a local maximum at $(x_0,t_0)$ in $\Omega \cap \{t \le t_0\}$.
Now suppose towards a contradiction that
\[\alpha := \min(-\Delta \phi - \Delta \Phi, \phi_t - |\nabla \phi|^2 - \nabla \phi \cdot \nabla \Phi)(x_0, t_0) > 0.\]
Then since $\phi$ and $\Phi$ are smooth, it follows that for a small $r > 0$ 
$$
\min(-\Delta \phi - \Delta \Phi, \phi_t - |\nabla \phi|^2 - \nabla \phi \cdot \nabla \Phi) > 2\alpha / 3 \quad\hbox{ in } C_r:=B_r(x_0) \times [t_0 - r, t_0].
$$

2.  Let $\Gamma$ be the parabolic boundary of $C_r$. We claim that $\Gamma \cap \Omega \subset \{\phi \ge \delta_0\}$ for some $\delta_0 > 0$.
\newline
To see this, suppose the claim is false: this means that we can find $(y,s) \in \Gamma \cap \Omega \cap \{\phi > 0\}^c$.  But then $\phi(y,s) \le 0$ and so $(\tilde{u}_1 - \phi)(y,s) \ge -\phi(y,s) \ge 0$ which violates the assumption that $\tilde{u}_1 - \phi$ is strictly negative in $\Gamma \cap C_r$.

3. Now we proceed to show that $u_m < \phi$ on the relevant part of the parabolic boundary, that is, there exists some $\e>0$ independent of $m$ such that 
\begin{equation}\label{order10}
u_m < \phi - \e \hbox{ on }   \Gamma \cap \supp(u_m) \hbox{ for sufficiently large } m.
\end{equation}
 To show \eqref{order10},  suppose not.  Then we can find $(x_k,t_k) \in \Gamma \cap \supp(u_{m_k})$ where $u_{m_k}(x_k,t_k) \ge \phi(x_k,t_k) - \frac{1}{k}$, and by compactness we can assume $(x_k,t_k) \to (x',t') \in \Gamma\cap\Omega$. 
  Then we have that for each $n$, there is $K(n)$ so that $|(x_k,t_k) - (x',t')| < 1/n$ and $m_k \ge k$ if $k \ge K(n)$, where we can assume $K(n) \ge n$.  Then
\[\sup_{k \ge K(n)} u_{m_k}(x_k,t_k) \le \sup_{\substack{ k \ge K(n) \\|(x',t') - (y,s)| < 1/n }} u_{k}(y,s) \le \sup_{\substack{k \ge n \\|(x',t') - (y,s)| < 1/n }} u_{k}(y,s).\]
Taking the infimum over both sides, we find
\[u_1(x',t') \ge \limsup_{k \to \infty} u_{m_k}(x_k,t_k) \ge \limsup_{k \to \infty} \phi(x_k,t_k) = \phi(x',t') > \delta_0,\]
which contradicts that $u_1 - \phi < 0$ on $\Gamma \cap \Omega$.
\medskip

\begin{figure}
\centerline{\includegraphics[width=4in]{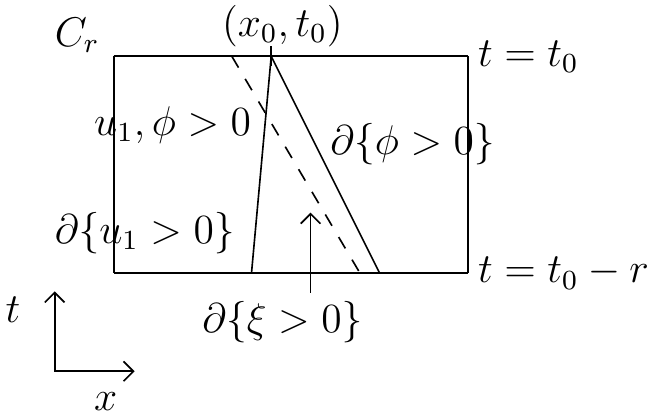}}
\caption{The motivation for $\xi$: it crosses $u_1$ at an earlier time}
\label{xiCrossesEarlier}
\end{figure}

4. Now  we define
\[\xi(x,t) = \phi(x+\gamma \nu, t), \hbox{ where } \nu = - \frac{\nabla \phi}{ | \nabla \phi|}(x_0,t_0).\]
Here $\gamma > 0$ is chosen small enough to satisfy first that 
\[
\min(-\Delta \xi - \Delta \Phi, \xi_t - |\nabla \xi|^2 - \nabla \xi \cdot \nabla \Phi) \geq 2\alpha/3+O(\gamma) >\alpha / 3\hbox{ in }C_r.
\]
  and secondly that for $m$ large, $u_m < \xi$ on $\Gamma \cap \supp(u_m)$ (which is possible since $\xi - \phi = O(\gamma)$ and $u_m - \phi$ is bounded away from zero on $\Gamma \cap \supp(u_m)$).
  \newline\newline
This justifies the following definition:
\[\tau_m := \sup\{t: (u_m - \xi)(x,t) < 0 \mbox{ for all } x \in \overline{\{u_m(\cdot, t)>0\}} \cap C_r\}.\]
Then $\tau_m$ will be the first crossing time of $u_m$ and $\phi$, provided they cross (since $u_m$ is continuous, we need not worry about jumps inside its support).  

5. We now wish to show that, along a subsequence, $u_m$ crosses $\xi$ in $C_r$.  To do this, we first prove that there is a subsequence $\{m_k\}$ so that
\[C_\delta \cap \{t < t_0\} \cap \supp(u_{m_k}) \ne \emptyset \mbox{ for all } \delta > 0.\] 

To show this first observe that, since $u_1(x_0,t_0) = 0$, there exists $M$ so that 
 \begin{align}
    \sup_{\substack{ m \ge M \\|(x_0,y_0) - (y,s)| < 1/M }} u_m(y,s) < 1. \label{boundOnu}
 \end{align}
Now assume towards a contradiction that there is a $\delta$ where our claim fails; we can take $\delta$ small enough so that $\delta < M^{-1}$.  We use Theorem~\ref{pmeBddSpeed} to derive our contradiction.  First, we use the theorem to find positive numbers $r_{\max}, T$ that depend on $K = 1$, the behavior of $\Phi$ near $(x_0,t_0)$, and dimension, and we may assume $T < \delta / 8$.  Now set $r_0$ a number smaller than $\min(r_{\max}, \delta/8)$ and $\mu =\min(T/8,r_0/8)$.  By definition of $\Omega$, we can find $(x',t')$ within distance $\mu$ of $(x_0,t_0)$ where $\eta(x',t) = 1$.  Thus we can find a subsequence $\{m_k\}_{k = 1}^\infty$ all bigger than $M$ and points $(y_{m_k},s_{m_k}) \in \supp \;u_{m_k}$ within distance $\mu$ of $(x',t')$.  

Consider a specific $m_k$.  Then by the assumption that $u_{m_k} = 0$ in $C_\delta \cap \{t < t_0\}$, since $r_0$ and $T_0$ are chosen much smaller than $\delta$, we find $u_{m_k} = 0$ in $B_{r_0}(x') \cap \{t = t_0 - T/2\}$.  Further, by \eqref{boundOnu}, 
$$u_{m_k} \le 1\hbox{ on the parabolic boundary of } B_{2r_0}(x') \times [t',t'+T].
$$
  Thus we apply Theorem~\ref{pmeBddSpeed} to find that 
\[u_{m_k} = 0 \in B_{r_0/4}(x') \times [t_0 -T/2,t_0+T/2]\]
and by the size of $r_0$, we have that $y_{m_k} \in B_{r_0/4}(x')$ and $s_{m_k} < t_0 + T/2$.  This yields $u_{m_k} = 0$ in a neighborhood of $(y_{m_k},s_{m_k})$.  This is a contradiction, so we find our claim holds for every $m_k$, giving us our desired subsequence.

6.  We will use the subsequence from the previous step to show that $\tau_{m_k} < t_0$. Indeed note that, for $|\alpha|, \beta > 0$,
\begin{eqnarray*}
  \xi(x_0+\alpha, t_0 - \beta)
& = & -\gamma  |\nabla \phi(x_0,t_0)| + \alpha \cdot \nabla \phi(x_0, t_0)  \\
&&- \beta \phi_t(x_0, t_0) + O(|\alpha|^2 + \beta^2 + \alpha \beta + \gamma^2).
\end{eqnarray*}
Thus there exists $\delta=\delta(\gamma) > 0$ so that if $|\alpha|^2 +\beta^2 < \delta^2$ then $\xi(x_0 + \alpha, t_0 - \beta) < 0$.  Due to the previous claim we can now find points $(y_{m_k},s_{m_k}) \in C_\delta \cap \{t < t_0\} \cap \supp(u_{m_k})$, hence
\[(u_{m_k}-\xi)(y_{m_k},s_{m_k}) > 0.\]
Thus $\tau_{m_k} \le s_{m_k} < t_0$. 

7.  Consequently there exists a crossing point $(x_{m_k}, \tau_{m_k}) \in C_r \cap \{t < t_0\} \cap \supp(u_{m_k})$ where $(u_{m_k}-\xi)(x_{m_k},\tau_{m_k}) = 0$.  Further, since $u_{m_k} - \xi < 0$ on $\Gamma\cap \text{supp }(u_{m_k})$ from step 3, we have that $(x_{m_k}, \tau_{m_k})$ is on the parabolic interior of $C_r$. Then we have that
\[\min(-\Delta \xi - \Delta \Phi, \xi_t - |\nabla \xi|^2 - \nabla \xi \cdot \nabla \Phi)(x_{m_k},\tau_{m_k}) > \alpha / 3\]
which forces that $(x_{m_k},\tau_{m_k}) \in \partial\{u_{m_k} > 0 \}$.  But then the inequality 
\[\left[\xi_t - |\nabla \xi|^2 - \nabla \xi \cdot \nabla \Phi\right](x_{m_k},\tau_{m_k}) > 0 \]
contradicts Lemma \ref{PMEDfact}, which applies since $\xi$ is smooth and thus satisfies \eqref{bdryTouchingCond}  at $(x_{m_k}, \tau_{m_k})$.

{\textit C. $\tilde{u}_1, u_2$ converges to $u_0$ at $t=0$}

It is not hard to check via comparison with radial barriers of (PME)$_m$, based on the local Lipschitz geometry of $\partial\Omega_0$,  that $\Omega(t)$ and $\{u_2(\cdot,t)>0\}$ converges to $\Omega_0$ in Hausdorff distance as $t\to0^+$.  From this fact and that $u_0$ solves $-\Delta u_0=\Delta\Phi$ in the interior of $\Omega_0$,  we have $\lim_{\tau\to 0}\tilde{u}_1(\cdot,\tau) = u_0$ from the definition of $\tilde{u}_1$. On the other hand $u_2$ satisfies $-\Delta u_2 >\Delta \Phi$ in $\{u_2>0\}\cap\{t>0\}$ and thus we have $\liminf_{\tau\to 0} u_2 (\cdot,\tau) \geq u_0$. Since $u_2 \leq \tilde{u}_1$ by definition,  it follows that $u_2(\cdot,\tau)$ converges to $u_0$ as $\tau\to 0$ as well. 
\end{proof}

\subsection{Convergence of $u_m$ as $m\to\infty$}
Now let us fix a compact set $\Omega_0$ in $\R^d$ with Lipschitz boundary, 
and let $u_0$ be as given in \eqref{initial}.  Let $u_m$ be the viscosity solution of (PME)$_m$ with initial data $u_0$. If we knew that $\{u_m\}$ locally uniformly converges to a function $u$ as $m\to\infty$, then Theorem~\ref{pmelimit} would yield that $u$ is a viscosity solution of $(P)$. Unfortunately we do not know whether such convergence is true: due to the quasi-static nature of $(P)$, $u$ may not be continuous over time and this may complicate the convergence of $u_m$.  Thus we take the alternative approach to show the convergence of the support of $\{u_m>0\}$ (see Theorem~\ref{convergence} (b)). The proof relies on the fact that $\{u_m\}$ has a stability property obtained from the $L^1$ contraction of the corresponding density function $\rho_m$ given by \eqref{pressure}. Using this stability as well as the comparison principle (Theorem~\ref{thm:cp}) we will obtain the support of $u_m$ converges to that of the unique solution $u$ of $(P)$. From this result we then obtain the uniform convergence of $\rho_m$ to the characteristic function of $\Omega_t$ away from the boundary of $\Omega_t$ (Corollary~\ref{convergence2}).

\begin{theorem}\label{convergence}
Take $\Omega_0$ and $u_0$ as given above and let $\Phi$ satisfy $(\textbf{A1})$. Then the following hold:
\begin{itemize}
\item[(a)] There exists a unique evolution of compact sets $\{\Omega_t\}_{t>0}$ such that any viscosity solution $u$ of $(P)$ satisfies $\Omega_t= \overline{\{u(\cdot,t)>0\}}$ for each $t>0$.\\
\item[(b)] For each $t>0$, the Hausdorff distance $d_H(\Omega_t,\overline{\{u_m(\cdot,t)>0\}})$ goes to zero as $m\to\infty$, and $\limsup_{m\to\infty} u_m(\cdot,t) $ is uniformly bounded.
\end{itemize} 
\end{theorem}

\begin{proof}
1. The proof is based on the $L^1$-contraction property (see e.g. section 3.5 of \cite{v}), which states that for two weak solutions $\rho_1, \rho_2$ of \eqref{pme_drift}, the $L^1$ norm of their differences decreases in time. In terms of the pressure variable $p_i = \frac{m}{m-1}\rho_i^{m-1}$, this reads
\begin{equation}\label{contraction}
\|(p_1^{1/(m-1)}-p_2^{1/(m-1)})(\cdot,t)\|_{L^1(\R^d)} \leq \|(p_1^{1/(m-1)}-p_2^{1/(m-1)})(\cdot,0)\|_{L^1(\R^d)}.
\end{equation}

Let us fix the initial data $u_0$ and $v_0$ so that $u_0\prec v_0$. For each $m$, let $u_m$ and $v_m$ be respectively the viscosity solutions of \pme with their respective initial data $u_0$ and $v_0$. Let us consider $u_1$, $u_2$, $\tilde{u}_1$ as given in Theorem~\ref{pmelimit}, and let $v_1, v_2, \tilde{v}_1$ denote the corresponding  functions  given in Theorem~\ref{pmelimit} defined with $\{v_m\}$ instead of $\{u_m\}$.

Since $u_0\prec v_0$, Theorem~\ref{thm:cp} applies to $\tilde{u}_1$ and $v_2$, and so using Theorem ~\ref{pmelimit} yields that 
 $$
 u_1\leq \tilde{u}_1 \leq v_2.
 $$ On the other hand, \eqref{contraction} yields that
 $$
 \|(u_m^{1/(m-1)}-v_m^{1/(m-1)})(\cdot,t)\|_{L_1} \leq \| (u_0^{1/(m-1)} - v_0^{1/(m-1)})(\cdot,0)\|_{L^1}.
 $$
 
The above inequality and the fact that $u_m \leq v_m$ and $\tilde{u}_1 \leq v_2$ imply that
\begin{equation}\label{contraction_2}
|\{v_2(\cdot,t)>0\} - \{\tilde{u}_1(\cdot,t)>0\}|\leq \limsup_{m\to\infty}  \|(u_m^{1/(m-1)}-v_m^{1/(m-1)})(\cdot,t)\|_{L_1} \leq |\{v_0>0\}-\{u_0>0\}|.
\end{equation}

2.   Take $u_0$ as given above, and let us consider
$$
V(x,t):= (\inf\{ v: v\hbox{ is a viscosity supersolution of } (P) \hbox{ with } u_0 \prec v(\cdot,0)\})_*
$$
and
$$
U(x,t):= \sup\{ u: u \hbox{ is a viscosity subsolution of } (P) \hbox{ with } u(\cdot,0)\prec u_0\}.
$$
Here $f_*$ denotes the lower semicontinuous envelop of $f$, as defined in \eqref{envelopes}.
Due to Theorem~\ref{thm:cp}, $U$ ($V$) then has the property of being below (above) any viscosity supersolution (subsolution) of $(P)$ with initial data $u_0$.
\medskip

 Let us consider a sequence of initial data $v^-_{0,n}$ and $v^+_{0,n}$ such that 
 \begin{itemize}
 \item[(a)] $v^{-,n}_0 \prec u_0 \prec v^{+,n}_0$ for each $n$;
 \item[(b)] $v^{\pm}_{0,n}$ uniformly converges to $u_0$ and $\{v^{\pm}_{0,n}>0\}$ converges $\{u_0>0\}$ uniformly in Hausdorff distance.
 \end{itemize}

  Such $v^{\pm}_{0,n}$ can be constructed using the fact that $\partial\Omega_0$ is locally Lipschitz. Now let $u^{\pm,n}_1$ and $u^{\pm,n}_2$ be the corresponding versions of $\tilde{u}_1$  and $u_2$ with the initial data $v^{\pm,n}_0$. Then due to Theorem~\ref{pmelimit} and  the definition of $U$ and $V$  we have 
 $$
 \tilde{u}^{-,n}_1 \leq U, \quad V\leq u^{+,n}_2 \hbox{ for any } n.
 $$

 Using these approximations of initial data, the fact that $\{V(\cdot,t)>0\}$ is open, and \eqref{contraction_2}, we conclude that 
 \begin{equation}\label{closures}
\Omega_t:=\overline{\{V(\cdot,t)>0\}} = \overline{\{U(\cdot,t)>0\}}.
 \end{equation}
 Now for any viscosity solution $u$ of $(P)$ with initial data $u_0$, we have $U \leq u \leq V$. Thus $\overline{\{u(\cdot,t)>0\}} = \Omega_t$, and we showed (a).

\medskip

3. By  Theorem~\ref{pmelimit} and the definition of $U$ and $V$,  we have 
$$
U\leq u_2  \leq (\tilde{u}_1)_* \leq V.
$$
Hence we have 
$$
\Omega_t=\overline{\{u_2(\cdot,t)>0\}}= \overline{\{\tilde{u}_1(\cdot,t)>0\}}.
$$ 
The above inequality and the fact that $\tilde{u}_1$ is a viscosity subsolution of $(P)$ with initial data $u_0$ yield (b).
\end{proof}
In terms of $\rho_m =  (\frac{m-1}{m}u_m)^{1/(m-1)}$ the convergence results can be stated as follows:

\begin{corollary}\label{convergence2}
Let $(\Omega_t)_{t>0}$ be the family of compact sets in $\R^d$ as given in Theorem~\ref{convergence}, and let  $\rho_m$ solve \eqref{pme_drift} with  initial data $\rho_m(\cdot,0) = (\frac{m-1}{m} u_0)^{1/(m-1)}$. Then  for each $t>0$,
\begin{itemize}
\item[(a)] $\limsup \rho_m \leq 1$;
\item[(b)] $\overline{\{\rho_m(\cdot,t)>0\}}$ uniformly converges to $\Omega_t$ in  Hausdorff distance; 
\item[(c)] $\rho_m(\cdot,t)$ locally uniformly converges to $1$ in $Int(\Omega_t)$, and to 0 in $(\Omega_t)^C$.
\end{itemize}
The same result holds for $\rho_m$ with initial data $\chi_{\Omega_0}$.
\end{corollary}

This concludes our analysis on the limiting profile of $\rho_m$. In the next two sections we study the gradient flow solution $\rho_{\infty}$ of the crowd transport equation \eqref{transport}. Among other things, we show that $\rho_m$ converge to $\rho_{\infty}$ as $m\to\infty$ in the Wasserstein distance (see Theorem ~\ref{main2}), and hence $\rho_\infty$ must coincide with $\chi_{\Omega_t}$.

\section{Convergence of the gradient flow solution as $m\to\infty$}
\label{sec:m_to_infty}

\subsection{Definition of the gradient flow solution and the discrete scheme}
For section 4 we introduce more assumptions:
\begin{itemize}
\item[(\textbf{A2})] $\displaystyle\inf_{\R^d} \Phi$ is finite, and without loss of generality we assume $\displaystyle\inf_{\R^d} \Phi = 0$.
\item[(\textbf{A3})] $\Phi$ is semi-convex, i.e. there exists $\lambda\in\R$ such that $D^2\Phi (x)\geq \lambda (Id)_{d\times d}$ for all $x\in\R^d$.\\

\item[(\textbf{A3'})] In addition to (\textbf{A3}), $\|\Delta \Phi\|_\infty \leq C$ for some finite $C$.
\end{itemize}

 The semi-convexity assumption (\textbf{A3}) guarantees the well-posedness of the discrete-time JKO solution.  When we prove convergence results as $m\to\infty$, we will replace (\textbf{A3}) by the stronger assumption (\textbf{A3'}). It ensures that $\Delta\Phi(x)$ cannot be too large, which makes it possible for us to obtain some quantitative estimates on the difference between  $\rho_m$ and $\rho_\infty$ for large $m$. $(\textbf{A2})$ is a technical assumption, and will be used explicitly in the proof of Lemma \ref{lemma:estimate_m_1} in section \ref{sec:m_to_infty}. The assumption $(\textbf{A1})$ will be only used in section 4 to link $\rho_{\infty}$ with the free boundary problem $(P)$, and in section 5 to obtain convergence results as $t\to\infty$. 
 
We denote by $\mathcal{P}_2(\mathbb{R}^d)$ the space of Borel probability measures on $\mathbb{R}^d$ with finite second moment, i.e., the set of probability measures $\rho(x)$ such that $\int_{\mathbb{R}^d} \rho(x) |x|^2 dx < \infty$. For a probability density $\rho \in \mathcal{P}_2(\mathbb{R}^d)$, we define its ``free energy'' $E_m[\rho]$ as
\begin{equation}
E_m[\rho] := \mathcal{S}_m[\rho] +\int_{ \mathbb{R}^d } \rho(x) \Phi(x) dx ~~\text{ for }1<m\leq \infty,
\label{def:E_m}
\end{equation}
where $\int_{ \mathbb{R}^d } \rho(x) \Phi(x) dx $ corresponds to the potential energy of $\rho$, and $\mathcal{S}_m[\rho]$ is its ``internal energy'', given by
\begin{equation}
\mathcal{S}_m[\rho] := 
\displaystyle\int_{\mathbb{R}^d} \frac{1}{m} \rho^m(x) dx~~ \text{ for }1<m<\infty,
\label{def:S_m}
\end{equation}
while $S_\infty$ is defined as
\begin{equation}
\mathcal{S}_\infty[\rho] := \begin{cases} 0 & \text{ for }\|\rho\|_{L^\infty(\mathbb{R}^d)} \leq 1\\
+\infty &\text{ otherwise. }\end{cases}
\label{def:S_infty}
\end{equation}

Next let us introduce the following discrete-time scheme (also called a minimizing movement scheme) introduced by \cite{jko}. We consider a time-step $h>0$, and initial data $\rho_0 \in \mathcal{P}_2(\mathbb{R}^d)$ satisfying $\|\rho_0\|_{L^\infty(\mathbb{R}^d)} \leq\nobreak 1$. For $1<m\leq \infty$, the sequence $(\rho_{m,h}^n)_{n\in \mathbb{N}}$ is recursively defined by $\rho_{m,h}^0 = \rho_0$ and 
\begin{equation}
\rho_{m,h}^{n+1} \in \arg\inf \left\{E_m[\rho] + \frac{1}{2h} W_2^2(\rho_{m,h}^n, \rho): \rho\in\mathcal{P}_2(\mathbb{R}^d)\right\},
\label{eq:def_jko}
\end{equation}
where $W_2(\cdot, \cdot)$ is the $2$-Wasserstein distance (for definition, see e.g. \cite{ags}.) We then define $\rho_{m,h}(x,t)$ as a function piecewise constant in time, given by
\begin{equation}
\rho_{m,h}(x,t) := \rho_{m,h}^{n}(x) \text{ for } t\in [nh, (n+1)h).
\label{eq:piecewise}
\end{equation}

Under the assumption in (\textbf{A3}) that $\Phi$ is semi-convex, one can check that for all $m>1$, the free energy $E_m[\rho]$ is $\lambda$-convex along the generalized geodesics with respect to 2-Wasserstein distance, where $\lambda$ is as given in (\textbf{A3}) (For the definition of generalized geodesics and $\lambda$-convexity, we refer to Appendix \ref{AppendixC}). One can then apply the theory of gradient flow solution developed in \cite{ags}, which gives the following existence and uniqueness results of the discrete solution, as well as a convergence result as $h\to 0$.

\begin{theorem}[\cite{ags}]\label{ags}
Let $1<m\leq \infty$ and suppose $\Phi$ satisfies \textup{(\textbf{A3})}. Moreover suppose $E_m[\rho_0]<\infty$, where $E_m$ be as given in \eqref{def:E_m}. Then for given $h>0$ the following holds for the sequence $(\rho_{m,h}^n)_{n\in\mathbb{N}}$ as defined in \eqref{eq:def_jko}: 
\begin{enumerate}[(a)]
\item Existence \& Uniqueness for discrete solutions (Section 2-3 of \cite{ags}): Let $\lambda$ be as defined in \textup{(\textbf{A3})}, and
let $h_0=-\frac{1}{\lambda}$ for $\lambda<0$, $h_0 = \infty$ for $\lambda\geq 0$. Then for $0<h<h_0$,  $\rho_{m,h}^n $ is uniquely defined for all $n\in \mathbb{N}$. 
\vspace{0.2cm}
\item Uniform convergence as $h\to 0$ (Theorem 4.0.7 -- 4.0.10 in \cite{ags}): Assume that $\Phi$ satisfies \textup{(\textbf{A2})}  in addition to \textup{(\textbf{A3})}, and consider initial data $\rho_0$ such that $E_m[\rho_0]\leq M$ for some constant $M$.  Let $\rho_{m,h}$ be as defined in \eqref{eq:piecewise}. Then for any $T>0$ and step size $0<h<1$, there exists some $\rho_m(t)$ (and $\rho_\infty(t)$ in the case $m=\infty$) in $C_W([0,T], \mathcal{P}_2(\mathbb{R}^d))$ such that
$$W_2(\rho_{m,h}(\cdot, t) , \rho_m(\cdot, t)) \leq C(\lambda) \sqrt{Mh} e^{-\lambda T} \text{ ~~for all }t\in [0,T],$$
where $\lambda$ is given by \textup{(\textbf{A3})}. 
Here we say $\rho\in C_W([0,T];\mathcal{P}_2(\R^d))$ if $\rho(\cdot,t)\in\mathcal{P}_2(\R^d)$ for each $0\leq t\leq T$ and 
$$
\rho(\cdot,t) \to \rho(\cdot,t_0)\hbox{ weakly in } \mathcal{P}_2(\R^d) \hbox{ as } t\to t_0 \hbox{ in }  [0,T].
$$

Moreover, for finite $m$, $\rho_m(x,t)$ coincides with the viscosity solution of \eqref{pme_drift}. 
\vspace{0.2cm}
\item Contraction in Wasserstein distance (Theorem 4.0.4 (iv) in \cite{ags}):  For a given $m$, consider the initial data $\rho_{01}, \rho_{02} \in \mathcal{P}_2(\mathbb{R}^d)$, with  $E_m[\rho_{0i}] < \infty$ for $i=1,2$. Let $\rho_1(x,t)$ and $\rho_2(x,t)$ denote the limit solutions as defined in part (b), with initial data
$\rho_{01}$ and $\rho_{02}$ respectively. Then we have the following stability result, where $\lambda$ is as given in \textup{(\textbf{A3})}:
$$W_2(\rho_1(\cdot,t), \rho_2(\cdot,t)) \leq e^{-\lambda t} W_2(\rho_{01}, \rho_{02}) \text{ for all }t\geq 0.$$

\end{enumerate}
\label{thm:collection}
\end{theorem}
The above theorem yields the gradient flow solutions $\rho_m(\cdot, t)$ and $\rho_\infty(\cdot, t)$. In this section, our main goal is to prove that as $m\to\infty$, $\rho_m(\cdot, t)$ converges to $\rho_\infty(\cdot, t)$ uniformly in $t\in[0,T]$ in 2-Wasserstein distance.  Convergence rates will also be obtained in terms of $m$. Although the rate is not optimal, to the best of our knowledge, our result is the first that gives some explicit convergence rate as the exponent $m\to\infty$ in the porous medium equation. More precisely, our main theorem in this section is as follows:

\begin{theorem}\label{main2}
Let $\Phi$ satisfy \textup{(\textbf{A2})} and \textup{(\textbf{A3'})}, and consider $\rho_0 \in \mathcal{P}_2(\mathbb{R}^d)$ satisfying $\|\rho_0\|_\infty \leq 1$ and $\int \rho_0 \Phi dx \leq M$.  Let $\rho_m(t)$ and $\rho_\infty(t)$ be as given in Theorem \ref{thm:collection}(b) with the initial data $\rho_0$. Then for any $T>0$, we have
$$\lim_{m\to \infty} \sup_{t\in[0,T]}W_2(\rho_m(t), \rho_\infty(t)) = 0.$$
More precisely, we have the following convergence rate:
$$\sup_{t\in[0,T]}W_2(\rho_m(t), \rho_\infty(t)) \leq \frac{C(M,T,\|\Delta \Phi\|_\infty)}{m^{1/24}}$$
\end{theorem}

We point out that under the additional assumption (\textbf{A1}) and the assumption that $\rho_0 = \chi_{\Omega_0}$, we can combine the results in Theorem \ref{main2} with Theorem \ref{convergence}, and immediately obtain that $\rho_\infty$ must coincide with $\chi_{\Omega_t}$ almost everywhere, which gives Theorem \ref{thm:coincide}.  Without these two additional assumptions, Theorem \ref{main2} still holds, but  our approach fails to yield the connection between $\rho_\infty$ with the free boundary problem (P). Thus a further characterization of $\rho_{\infty}$ beyond as a weak solution of \eqref{weak_eqn} remains open in the general context.

The rest of this section will be devoted to proving Theorem \ref{main2}. In section \ref{subsec:one-step}, we consider the discrete JKO scheme \eqref{eq:def_jko} for $E_m$ and $E_\infty$ respectively, with the same initial data $\|\rho_0\|_\infty \leq 1$. We show that if we run the JKO scheme for one step only, then their Wasserstein distance is small.
Once we have the one-step estimate, we are finally ready to prove Theorem \ref{main2} in section \ref{subsec:proof_main2}, which says the Wasserstein distance between the continuous gradient flow solutions $\rho_m$ and $\rho_\infty$ also goes to zero as $m\to\infty$, with an explicit rate in terms of $m$.

\subsection{One-step estimate for large $m$}
\label{subsec:one-step}
We consider the initial data $\rho_0 \in \mathcal{P}_2(\mathbb{R}^d)$ satisfying $\|\rho_0\|_\infty \leq 1$ and with finite potential energy, and let $h$ be some fixed small time step. Then for any $2<m\leq \infty$, we define $\mu_m$ (and $\mu_\infty$ in the case $m=\infty$) as follows:
\begin{equation}
\mu_m := \underset{\rho \in \mathcal{P}_{2}(\mathbb{R}^d)} {\argmin}\left[ E_m[\rho] +  \frac{1}{2h} W_2^2(\rho_0, \rho) \right].
\label{eq:one_step_rho_m}
\end{equation}
Our main result in this subsection is Proposition \ref{prop:one_step_m_infty}, which says that the Wasserstein distance between $\mu_m$ and $\mu_\infty$ is of order $O(m^{-1/8})$ for large $m$.  To show that we first establish the following two technical lemmas concerning $\mu_m$ for $2<m<\infty$.

\begin{lemma} Let $2<m<\infty$, and let $\Phi$ satisfy \textup{(\textbf{A2})} and \textup{(\textbf{A3})}, and consider the initial data $\rho_0 \in \mathcal{P}_2(\mathbb{R}^d)$ satisfying $\|\rho_0\|_\infty \leq 1$ and $\int \rho\Phi \leq M$. Letting $\mu_{m}$ be defined as in \eqref{eq:one_step_rho_m}, the following estimate holds (where $a_+ := \max\{a, 0\}$):
$$\int_{\mathbb{R}^d} (\mu_m - 1)_+ dx \leq  2\sqrt{\frac{M+1}{m}}.$$
\label{lemma:estimate_m_1}
\end{lemma}

\begin{proof} 
Our proof is based on the following crude estimate: $\frac{1}{m} \int_{\mathbb{R}^d} (\mu_m)^m dx \leq M+1$. This inequality directly comes from the fact that $E_m[\mu_m] \leq E_m[\rho_0]$, together with the assumption (\textbf{A2}) that $\inf \Phi \geq 0$:
\begin{equation}
\int_{\mathbb{R}^d} \frac{1}{m} (\mu_m)^m dx \leq   \int_{\mathbb{R}^d} \rho_0 \Phi dx + \int_{\mathbb{R}^d} \frac{1}{m} (\rho_0)^m dx  \leq M+1,
\label{ineq:rho_m_rho_0}
\end{equation}
which upon rearranging gives
\begin{equation}
\int_{\mathbb{R}^d} (\mu_m)^m dx \leq m (M+1).
\label{ineq:A_2}
\end{equation}
Note that for $m> 2$, we have
\begin{equation}
\int_{\{\mu_m\geq1\}} (\mu_m)^m dx \geq \int_{\{\mu_m\geq1\}} \Big( 1 + m(\mu_m-1) + \frac{m(m-1)}{2}(\mu_m-1)^2\Big) dx.
\label{ineq:integral_rho>1}
\end{equation}
Combining the inequalities \eqref{ineq:A_2} and \eqref{ineq:integral_rho>1} together, we have
$$\int_{\mathbb{R}^d} (\mu_m - 1)_+^2 dx \leq \frac{2(M+1)}{m-1} \leq \frac{4(M+1)}{m}.$$
Finally, note that $|\{x: \mu_m(x) \geq 1\}|\leq 1$, and so the Cauchy-Schwarz inequality yields that
\[\int_{\mathbb{R}^d} (\mu_m - 1)_+ dx \leq 2\sqrt{\frac{M+1}{m}}.\qedhere\]
\end{proof}

The following lemma says that for large $m$, we can find a probability density $\tilde \mu_m$ that is close to $\mu_m$ in Wasserstein distance, has maximum density bounded by one, and has potential energy not much larger than $\mu_m$.

\begin{lemma}
\label{lemma:estimate_m_2}
Let $\Phi$ satisfy \textup{(\textbf{A2}) and (\textbf{A3'})}. Under the conditions of Lemma \ref{lemma:estimate_m_1}, there exists a probability density $\tilde \mu_m \in \mathcal{P}_2(\mathbb{R}^d),$ such that $\|\tilde\mu_m\|_{L^\infty(\mathbb{R}^d)} \leq 1$, 
\begin{equation}
\int_{\mathbb{R}^d}  \tilde \mu_m \Phi dx \leq \int_{\mathbb{R}^d} \mu_m\Phi dx+2 \|\Delta \Phi\|_\infty \sqrt{\frac{M+1}{m}},
\label{ineq:energy_difference}
\end{equation}
and $\tilde \mu_m$ is ``close'' to $\mu_m$ in the sense that
\begin{equation}
W_2(\mu_m, \tilde \mu_m) \leq \frac{2(M+1)^{1/4}}{m^{1/4}}.
\label{ineq:w2_small}
\end{equation}
\end{lemma}

\begin{proof}
Due to the previous lemma, $\int_{\mathbb{R}^d} (\mu_m - 1)_+ dx \leq 2\sqrt{\frac{M+1}{m}}$ for all $2<m<\infty$ and $h>0$. We denote by $a:= 2\sqrt{\frac{M+1}{m}}$ for short, and note that $a$ is small for large $m$. Next we will give a explicit construction of  $\tilde \mu_m$, such that it satisfies all the requirements.

We begin with breaking $\mu_m$ into the sum
 $$\mu_m(x) = \mu_m^1(x) +\mu_m^2(x),$$ where
$$\mu_m^1(x) := \min\{\mu_m(x), 1-a\}, \quad\quad \mu_m^2(x) := \big(\mu_m(x) - (1-a)\big)_+.$$

The idea is to construct $\tilde \mu_m$ by keeping $\mu_m^1$ and modifying $\mu_m^2$.  We first make the observation that $\mu_m^2$ only contains a small amount of mass: more precisely, 
\begin{equation}
\int_{\mathbb{R}^d} \mu_m^2(x)dx \leq 2a.
\label{ineq:eta_2}
\end{equation} This is due to the following two facts.  First, due to  Lemma \ref{lemma:estimate_m_1}, the mass of $\mu_m$ above 1 cannot exceed $a$. Second, we claim $|\{\mu_m >1-a\}| \leq 1$. To show the claim, suppose not, then we have $\int \min\{\mu_m, 1-a\}dx > 1-a$. As a result, $\int \mu_m dx > (1-a) + a =1$, where the $(1-a)$ corresponds to the mass below $(1-a)$, and $a$ corresponds to the mass squeezed between $(1-a)$ and $1$ due to our (false) assumption.

Let us now construct $\tilde \mu_m$ as follows:
\begin{equation}
\tilde \mu_m(x) := \mu_m^1(x) + (g*\mu_m^2)(x),\label{def:tilde_rho_m}
\end{equation}
where * denotes convolution and  $g(x) := \frac{1}{2}\chi_{B(0,R)},$
where $R(d)$ is the dimensional constant chosen such that $\int_{\mathbb{R}^d} g(x) dx=1$. Note that although $R(d)$ depends on $d$, we indeed have $R(d) \leq 1$ for all $d\geq 1$.

We claim that $\tilde \mu_m$ constructed in \eqref{def:tilde_rho_m} satisfies all the requirements stated in the theorem. First note that  the facts $\int g = 1$ and $g\geq 0$ imply that $\tilde\mu_m$ is nonnegative and has the same mass as $\mu_m$. To show that $\|\tilde \mu_m\|_\infty \leq 1$, it suffices to check $\|g*\mu_m^2\|_\infty\leq a$. Since the mass of $\mu_m^2$ is less than $2a$, this inequality is a direct consequence of Young's inequality: 
$$\|g*\mu_m^2\|_\infty \leq \|\mu_m^2\|_1 \|g\|_\infty \leq 2a \cdot \frac{1}{2} = a.$$

Next we verify that the inequality \eqref{ineq:energy_difference} holds, which is equivalent to
$$\int_{\mathbb{R}^d} (g*\mu_m^2) \Phi dx \leq \int_{\mathbb{R}^d}\mu_m^2 \Phi dx+\|\Delta \Phi\|_\infty a.$$
This can be rewritten as
$$\int_{\mathbb{R}^d} \mu_m^2(x) \Big[ (g*\Phi)(x) - \Phi(x)\Big] dx \leq \|\Delta \Phi\|_\infty a.$$
Since $\mu_m^2$ has mass less than $2a$, it suffices to show that
$$\left| \fint_{B(x,R(d))} \Phi(y) dy - \Phi(x) \right| \leq \frac{1}{2}\|\Delta \Phi\|_\infty \text{ for all }x\in\mathbb{R}^d,$$
where we used the fact that $R(d)\leq 1$ to get the right hand side, and note that $\|\Delta \Phi\|_\infty$ is finite due to (\textbf{A3'}). The proof of this inequality is similar to the proof of the mean value property for harmonic functions, and hence is omitted here.

Finally it remains to show \eqref{ineq:w2_small}, which is equivalent to
\begin{equation}
W_2(\mu_m, \tilde \mu_m)\leq \sqrt{2a}.
\label{ineq:w2_small_2}
\end{equation}
We now heuristically describe a transport plan, which is not necessarily optimal.  First, we keep the mass of $\mu_m^1$ at its original location, so that no transportation cost is induced.  Second, for for every ``particle'' located at $x$ in $\mu_m^2$, the transport plan is to distribute it evenly in the disk $B(x,R(d))$. (Again recall that $R(d)\leq 1$ for any dimension $d\geq 1$.) Since the mass of $\mu_m^2$ is no more than $2a$, the total cost of the transportation plan is bounded by $2aR(d)^2$, which immediately implies \eqref{ineq:w2_small_2}. 
\end{proof}

Now we are ready to state the following one-step estimation, which controls the Wasserstein distance between $\mu_m$ and $\mu_\infty$:

\begin{proposition}
\label{prop:one_step_m_infty}
Let  $\Phi$ satisfy \textup{(\textbf{A2})}  and \textup{(\textbf{A3'})}, and consider the initial data $\rho_0 \in \mathcal{P}_2(\mathbb{R}^d)$ with $\|\rho_0\|_\infty \leq 1$ and $\int \rho_0\Phi dx \leq M$. Let $\lambda$ be as given in \textup{(\textbf{A3})}, and $\lambda^- := -\min\{0,\lambda\}$. For any $0<h<\frac{1}{32\lambda^-}$, let $\mu_m$ and $\mu_\infty$ be as defined in \eqref{eq:one_step_rho_m} for the cases $m$ finite and $m=\infty$ respectively. 
Then the following inequality holds:
$$W_2(\mu_m, \mu_\infty) \leq \frac{1}{m^{1/8}}C(M).$$
\end{proposition}

\begin{proof}
Let us fix $M$ and $d$. Suppose the statement is false.  Then for an arbitrarily large $A_0 > 0$, there exist $m>2$, $0<h<\frac{1}{32\|\Delta\Phi\|_\infty}$  such that 
\begin{equation}
W_2(\mu_m, \mu_\infty) = A m^{-1/8}, \text{ where }A> A_0.\label{assumption:A}
\end{equation} 
To get a contradiction, we will construct a new probability measure $\eta \in \mathcal{P}_2(\mathbb{R}^d)$ with $\| \eta\|_{L^\infty(\mathbb{R}^d)} \leq 1$ such that the following inequality holds if $A_0$ is chosen to be sufficiently large:
\begin{equation}
\left[ E_m[\eta] +  \frac{1}{2h} W_2^2(\rho_0, \eta) \right] + \left[ E_\infty[\eta] +  \frac{1}{2h} W_2^2(\rho_0, \eta) \right] < \left[ E_m[\mu_m] +  \frac{1}{2h} W_2^2(\rho_0, \mu_m) \right] +  \left[ E_\infty[\mu_\infty] +  \frac{1}{2h} W_2^2(\rho_0, \mu_\infty) \right] . \label{ineq:eta}
\end{equation}
This means that $\eta$ would beat at least one of the minimizers in \eqref{eq:one_step_rho_m} for some $m$ ($m$ may either be finite or $+\infty$), contradicting the definition of $\mu_m$ and $\mu_\infty$.

The probability density $\eta$ is constructed as follows.  Let $\tilde \mu_m$ be the probability density constructed in Lemma \ref{lemma:estimate_m_2},  and we denote by $\tilde{T}_m$ the optimal transport map such that $(\tilde{T}_m)_\# \rho_0 = \tilde \mu_m$.  Similarly, let $T_\infty$ be the optimal transport map such that $(T_\infty)_\# \rho_0 = \mu_\infty$. Then $\eta$ is defined as \begin{equation}
\eta = \Big(\frac{1}{2} \tilde T_m + \frac{1}{2} T_\infty\Big) \# \rho_0.
\label{def:eta}
\end{equation} 
 $\eta$ is thus the midpoint between $\tilde\mu_m$ and $\mu_\infty$ on their generalized geodesics, as defined in Sec 9.2 of \cite{ags}.

Next we will prove that $\eta$ satisfies the inequality \eqref{ineq:eta}. This is done by proving the inequalities \eqref{ineq:entropy_eta}--\eqref{ineq:distance_eta}:
\begin{eqnarray}
\mathcal{S}_m[\eta] + \mathcal{S}_\infty[\eta] &\leq& \mathcal{S}_m[\mu_m] + \mathcal{S}_\infty[\mu_\infty] + \dfrac{1}{m}, \label{ineq:entropy_eta}\\
2\int_{\mathbb{R}^d} \eta \Phi dx &\leq& \int_{\mathbb{R}^d}\mu_m \Phi dx + \int_{\mathbb{R}^d}\mu_\infty \Phi dx + \dfrac{C(M)(\|\Delta \Phi\|_\infty-\lambda) }{\sqrt{m}} - \frac{2A^2\lambda}{m^{1/4}}, \label{ineq:energy_eta}\\
\frac{1}{h} W_2^2(\rho_0, \eta) &\leq& \frac{1}{2h} W_2^2(\rho_0, \mu_m) + \frac{1}{2h} W_2^2(\rho_0, \mu_\infty) +\frac{C(M)- A^2/8 }{h m^{1/4}}. \label{ineq:distance_eta}
\end{eqnarray}
If $A_0$ is chosen to be a sufficiently large number depending only on $M$ and $d$,  since $A>A_0$, the sum of these three  inequalities implies \eqref{ineq:eta} (where we make use of the assumption that $h\leq \frac{1}{32\lambda^-}$), thereby yielding a contradiction.

To show \eqref{ineq:entropy_eta}, it suffices to prove that $\|\eta\|_\infty \leq 1$, since then $S_m[\eta]= \frac{1}{m}\int \eta^m dx \leq \frac{1}{m}$ and $S_\infty[\eta] = 0$. Proposition 9.3.9 of \cite{ags} states that the $L^p$ norm with $p>1$ is convex along the generalized geodesics, and thus 
$$2\|\eta\|_p \leq \|\tilde \mu_m\|_p + \|\mu_\infty\|_p\hbox{ for all }p>1.
$$ Sending $p\to \infty$ in the above inequality yields $\|\eta\|_\infty \leq 1$, since both $\|\tilde \mu_m\|_\infty$ and $\|\mu_\infty\|_\infty$ are bounded by 1. 

\eqref{ineq:energy_eta} comes from the semi-convexity of $\Phi$ given by (\textbf{A3}) (which is a consequence of (\textbf{A3'})). Let $\lambda \in \mathbb{R}$ be as given in (\textbf{A3}). Proposition 9.3.2 in \cite{ags} yields that $\int \rho\Phi dx$ is a $\lambda$-convex functional of $\rho$ along any generalized geodesic, and thus
\begin{equation}
\begin{split}
2\int_{\mathbb{R}^d} \eta \Phi dx &\leq \int_{\mathbb{R}^d} \tilde \mu_m \Phi dx + \int_{\mathbb{R}^d}  \mu_\infty \Phi dx - \frac{1}{4}\lambda W_2^2(\tilde \mu_m, \mu_\infty)\\
&\leq \int_{\mathbb{R}^d} \mu_m \Phi dx + \int_{\mathbb{R}^d}  \mu_\infty \Phi dx + \frac{C(M)\|\Delta \Phi\|_\infty }{m^{1/2}} - \frac{1}{4} \lambda W_2^2(\tilde \mu_m, \mu_\infty),
\label{ineq:energy_eta_1}
\end{split}
\end{equation}
where the last line comes from \eqref{ineq:energy_difference}. Next let us estimate $W_2^2(\tilde \mu_m, \mu_\infty)$.  Due to our assumption \eqref{assumption:A} in the beginning of this proof and the inequality \eqref{ineq:w2_small}, we have
\begin{equation*}
W_2(\tilde\mu_m, \mu_\infty)\leq W_2(\mu_m, \mu_\infty) + W_2(\mu_m, \tilde \mu_m)\leq Am^{-1/8} + C(M)m^{-1/4}.
\end{equation*}
We take the square of the above inequality and apply the inequality $(a+b)^2 \leq 2a^2 + 2b^2$, and obtain
$$W_2^2(\tilde\mu_m, \mu_\infty) \leq 2A^2m^{-1/4} + C(M)m^{-1/2}.
$$
 Plugging this inequality into \eqref{ineq:energy_eta_1} yields \eqref{ineq:energy_eta}.

Finally it remains to show \eqref{ineq:distance_eta}. Due to Lemma 9.2.1 of \cite{ags} $W_2^2(\rho_0,\cdot)$ is 1-convex along generalized geodesics, and thus
\begin{equation}
W_2^2(\rho_0, \eta) \leq ~\frac{1}{2}W_2^2(\rho_0, \tilde \mu_m) + \frac{1}{2} W_2^2(\rho_0, \mu_\infty) - \frac{1}{4}W_2^2(\tilde \mu_m, \mu_\infty).
\label{ineq:distance}
\end{equation}
Now, by the triangle inequality, we have
\begin{equation*}
\begin{split}
W_2^2(\rho_0, \eta) \leq &~\frac{1}{2} \Big( W_2(\rho_0, \mu_m) + W_2(\mu_m, \tilde \mu_m) \Big)^2 + \frac{1}{2}  W_2^2(\rho_0, \mu_\infty) - \frac{1}{4}\Big( W_2(\mu_m, \mu_\infty) - W_2(\mu_m, \tilde \mu_m )\Big)^2\\
\leq & ~\frac{1}{2}W_2^2(\rho_0, \mu_m) + \frac{1}{2}W_2^2(\rho_0, \mu_\infty) - \frac{1}{4} W_2^2(\mu_m, \mu_\infty) + \frac{C(M)}{2m^{1/4}} (1+ W_2(\mu_m, \mu_\infty))\\
\leq & ~\frac{1}{2}W_2^2(\rho_0, \mu_m)^2 + \frac{1}{2}W_2^2(\rho_0, \mu_\infty) + \frac{C(M) - A^2/8}{m^{1/4}}.
\end{split}
\end{equation*}

For the second inequality we used the fact that $W_2(\mu_m, \tilde \mu_m) \leq C(M)m^{-1/4}$ due to Lemma \ref{lemma:estimate_m_2} as well as that $W(\rho_0, \mu_m) \leq C(M)$  for all $m$ and $h$ (otherwise $\mu_m$ would fail to be a minimizer of $\Psi_m$).  For the third inequality we use the assumption that $W_2(\mu_m, \mu_\infty) = A m^{-1/8}$. Finally, dividing both sides of the above inequality by $h$ yields \eqref{ineq:distance_eta}.
\end{proof}

\subsection{Convergence of the continuum solutions as $m\to\infty$}
\label{subsec:proof_main2}
In this subsection, we give a proof of Theorem~\ref{main2}. The proof is done by combining the one-step estimation results in section \ref{subsec:one-step} with the convergence results for discrete solutions as $h\to 0$.
\begin{proof}[\textbf{Proof of Theorem \ref{main2}}]
1.  Note that the assumptions on $\rho_0$ immediately imply that $E_m[\rho_0] \leq M+1$ for all $2<m\leq \infty$. This enables us to apply Theorem \ref{thm:collection}(b): For all time steps $h$ satisfying $0<h<h_0$ (where $h_0$ is a small constant depending on $M, T, \|\Delta \Phi\|_\infty$), we have
\begin{equation}
W_2(\rho_m(t), \rho_{m,h}(t)) \leq C(\lambda)\sqrt{M} e^{-\lambda T}\sqrt{h}\\
=: C\sqrt{h}\text{ for all }t\in[0,T],
\label{ineq:time_step}
\end{equation}
and this inequality holds for both finite $m$ and $m=\infty$.
For notational simplicity, the various constants $C$ appearing in this proof may depend on $M, \| \Delta \Phi\|_\infty, T$, and the value of $C$ may differ from line to line. 

2. Now we fix the small time step $h$ such that $0<h<h_0$, and our goal is to show that  
\begin{equation}
W_2(\rho_{m,h}(t), \rho_{\infty, h}(t)) \leq C\sqrt{h} \text{ for all } t\in[0,T] \text{ when }m \geq Ch^{-12}.
\label{ineq:step2}
\end{equation}  Before we prove this inequality, let us point out the proof is finished once we obtain this: by combining \eqref{ineq:step2} with the inequality \eqref{ineq:time_step} (and note that \eqref{ineq:time_step} holds for both finite $m$ case and $m=\infty$), one immediately has
$$W_2(\rho_{m}(t), \rho_{\infty}(t)) \leq C\sqrt{h}  \text{ for all }t\in[0,T] ~~~\text{ given that }m\geq Ch^{-12},$$
which concludes the proof and would give the rate $W_2(\rho_{m}(t), \rho_{\infty}(t)) \leq C m^{-1/24}$ for $t\in[0,T]$.

3. To prove the inequality \eqref{ineq:step2} in step 2, note that it is equivalent to prove
 \begin{equation}\label{claim}
 W_2(\rho_{m, h}^n, \rho_{\infty, h}^n)\leq C\sqrt{h} \text{ ~~for all }n\leq \frac{T}{h}.
 \end{equation}

From now on we will denote $\rho_{m,h}^n$  by $\rho_m^n$ (and denote $\rho_{\infty, h}^n$ by $\rho_\infty^n$) for notational simplicity.  Proposition \ref{prop:one_step_m_infty} then shows that $W_2(\rho_m^1, \rho_\infty^1)$ is small for sufficiently large $m$. To deal with the case $n>1$ we consider the tree structure as illustrated in Figure \ref{fig:tree}.

\begin{figure} 
\begin{tikzpicture}[>=latex, thick]
\tikzstyle{line} = [draw, -latex]
\tikzstyle{arrow2} = [draw, latex-latex, blue!70]
\tikzstyle{arrow3} = [draw, latex-latex, orange]
\tikzstyle{white_bg}=[rectangle,fill=white]

\pgfmathsetmacro{\h}{2.2}
\pgfmathsetmacro{\w}{1}
\pgfmathsetmacro{\y}{0.25cm}
\pgfmathsetmacro{\x}{0.3cm}

\draw (-5, 0) node[black] (rho_0) {\Large{$\rho_0$}};
\draw (rho_0)+(\h,\w) node[black] (rho_m1) {\Large{$\rho_m^1$}};
\draw (rho_m1)+(\h,\w) node[black] (rho_m2) {\Large{$\rho_m^2$}};
\draw (rho_m2)+(\h,\w) node[black] (rho_m3) {\Large{$\rho_m^3$}};
\draw (rho_m3)+(\h,\w) node[black] (rho_m4) {\Large{$\cdots$}};

\path [line] (rho_0) -- node[pos = 0.45, yshift=\y, black] {$m$} (rho_m1) ;
\path [line] (rho_m1) -- node[pos = 0.45, yshift=\y, black] {$m$} (rho_m2) ;
\path [line] (rho_m2) -- node[pos = 0.45, yshift=\y, black] {$m$} (rho_m3) ;
\path [line] (rho_m3) -- node[pos = 0.45, yshift=\y, black] {$m$} (rho_m4);

\draw (rho_0)+(\h,-\w) node[black] (rho_i1) {\Large{$\rho_\infty^1$}};
\draw (rho_i1)+(\h,-\w) node[black] (rho_i2) {\Large{$\rho_\infty^2$}};
\draw (rho_i2)+(\h,-\w) node[black] (rho_i3) {\Large{$\rho_\infty^3$}};
\draw (rho_i3)+(\h,-\w) node[black] (rho_i4) {\Large{$\cdots$}};

\path [line] (rho_0) -- node[pos = 0.45, yshift=-\y, black] {$\infty$} (rho_i1) ;
\path [line] (rho_i1) -- node[pos = 0.45, yshift=-\y, black] {$\infty$} (rho_i2) ;
\path [line] (rho_i2) -- node[pos = 0.45, yshift=-\y, black] {$\infty$} (rho_i3) ;
\path [line] (rho_i3) -- node[pos = 0.45, yshift=-\y, black] {$\infty$} (rho_i4) ;

\draw (rho_i1)+(\h,\w) node[black] (eta_2) {\Large{$\eta^2$}};
\path [line] (rho_i1) -- node[pos = 0.45, yshift=\y, black] {$m$} (eta_2) ;

\draw (rho_i2)+(\h,\w) node[black] (eta_3) {\Large{$\eta^3$}};
\path [line] (rho_i2) -- node[pos = 0.45, yshift=\y, black] {$m$} (eta_3) ;

\draw (rho_i3)+(\h,\w) node[black] (eta_4) {\Large{$\cdots$}};
\path [line] (rho_i3) -- node[pos = 0.45, yshift=\y, black] {$m$} (eta_4) ;

\path [arrow3] (rho_i2) to  node[white_bg, xshift=-1.6, pos = 0.5] {$\leq \delta$} (eta_2);
\path [arrow3] (rho_i3) to  node[white_bg, xshift=-1.6, pos = 0.5] {$\leq \delta$} (eta_3);

\path [arrow2] (rho_m1) to [bend left=5] node[pos = 0.5, xshift=0.48cm] {$d_1$\textcolor{orange}{$\leq \delta$}} (rho_i1);
\path [arrow2] (rho_m2) to [bend left=17] node[pos = 0.5, xshift=\y] {$d_2$} (rho_i2);
\path [arrow2] (rho_m3) to [bend left=20] node[pos = 0.5, xshift=\y] {$d_3$} (rho_i3);
\end{tikzpicture}

\caption{Illustration of the tree structure. \label{fig:tree}}
\end{figure}
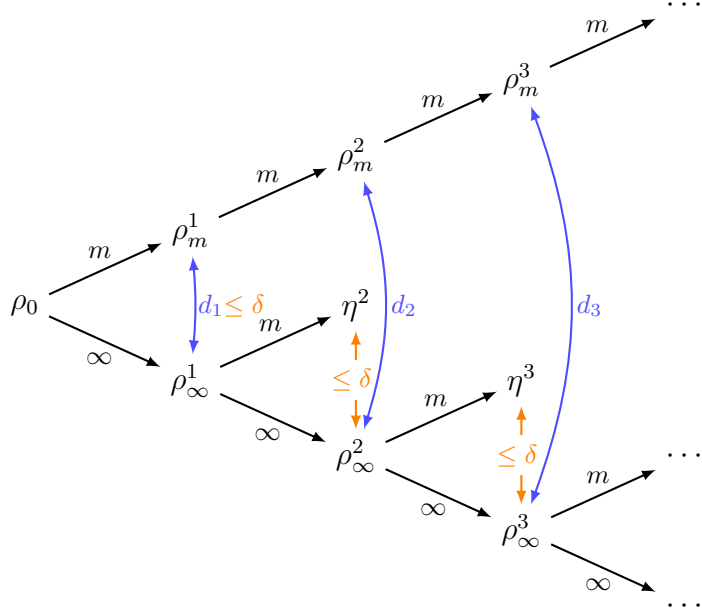
Here for $n\geq 2$, $\eta^n$ is defined as below:
$$\eta^n :=\underset{\rho \in \mathcal{P}_2(\mathbb{R}^d)}{\argmin} \left\{E_m[\rho] + \frac{1}{2h} W_2^2(\rho_\infty^{n-1}, \rho) \right\} ~~\text{ for } n\geq 2.$$
We point out that for $n\geq 2$, $\|\rho_\infty^{n-1}\|_\infty \leq 1$ holds by definition, and in addition we have $\int \rho_\infty^{n-1} \Phi dx \leq \int \rho_0\Phi dx \leq M$. Hence by taking $\rho_\infty^{n-1}$ as the initial data, the one-step estimate in Proposition \ref{prop:one_step_m_infty} yields that
\begin{equation}
W_2(\rho_\infty^n, \eta^n) \leq \delta := C(M) m^{-1/8} \text{ for all } n\geq 2.
\label{def:delta}
\end{equation}
Let us denote $d_n := W_2(\rho_m^n, \rho_\infty^n)$, which satisfies
\begin{equation}
d_n \leq W_2(\rho_m^n, \eta^n) + W_2(\eta^n, \rho_\infty^n)\leq W_2(\rho_m^n, \eta^n) + \delta,
\label{ineq:d_n1}
\end{equation}
where $\delta$ is as defined in \eqref{def:delta}. Now it remains to control $W_2(\rho_m^n, \eta^n)$. Note that $\rho_m^n$ and $\eta^n$ are minimizers given by the discrete-time scheme \eqref{eq:def_jko} with the same free energy functional $E_m$, but with different initial data $\rho_m^{n-1}$ and $\rho_\infty^{n-1}$. To estimate $d_n$ in terms of $d_{n-1}$, we use Lemma 4.2.4 of \cite{ags} which states that the Wasserstein distance between two discrete solutions does not grow too fast. More precisely, it gives the following inequality
\begin{equation}
\begin{split}
W_2^2(\rho_m^n, \eta^n) &\leq e^{-2 \lambda^- h}\left[W_2^2(\rho_m^{n-1}, \rho_\infty^{n-1}) + h \big( E_m[\rho_m^{n-1}] - E_m[\rho_m^{n}]\big)\right]\\
&\leq e^{-2 \lambda^- h}\left(d_{n-1}^2 + h a_{n-1}\right),
\end{split}
\label{ineq:d_n2}
\end{equation}
where $\lambda^- := -\min\{\lambda,0\}$, with $\lambda$ as given in \textup{(\textbf{A3})}. We denote $a_{n-1} := E_m[\rho_m^{n-1}] - E_m[\rho_m^n]$, which satisfies the following properties:
\begin{equation}
a_n \geq 0 \text{ for all }n \in \mathbb{N}^+, \text{ and }\sum_{n=0}^\infty a_n \leq M+1.
\label{assumption:a_n}
\end{equation}
Finally, we plug  \eqref{ineq:d_n2} into \eqref{ineq:d_n1} to obtain the following family of inequalities:
\begin{eqnarray}
\nonumber d_1 &\leq& \delta\\
d_n &\leq& e^{-2 \lambda^- h}\sqrt{d_{n-1}^2 + h a_{n-1} }+ \delta \quad\text{for }n=2,3,\ldots.
\label{ineq:dn}
\end{eqnarray}

4. We next focus on the inequality \eqref{ineq:dn}, and our goal is to show that $d_{\frac{T}{h}} \leq C\sqrt{h}$ for $\delta$ sufficiently small (more precisely, $\delta \leq h^{3/2} $ would be enough).  By taking the square of \eqref{ineq:dn} and applying the inequality $2ab \leq ha^2 + \frac{b^2}{h}$, we obtain
\begin{eqnarray}
\nonumber d_n^2 &\leq& (1+h)e^{-4\lambda^- h} (d_{n-1}^2 + h a_{n-1}) + (1+\frac{1}{h})\delta^2\\
&\leq&  (1+ C h) d_{n-1}^2 + h \underbrace{(2a_{n-1} + 2h)}_{:=b_{n-1}},
\label{ineq:dn_2}
\end{eqnarray}
where in the last line we let $\delta \leq h^{3/2}$ so that $(1+\frac{1}{h})\delta^2 \leq 2h^2$. Also note that $b_n := 2a_n + 2h$ satisfies $\sum_{n=0}^{T/h} b_n \leq 2(M+T+1)$. Now by dividing by $(1+Ch)^n$ on both sides of \eqref{ineq:dn_2} and summing the inequality from $2$ to $n$, we obtain that
$$d_n^2 \leq d_1^2 (1+Ch)^{n-1} + \sum_{k=1}^n h b_k (1+Ch)^{n-k} \text{ for all }n.
$$
Hence as a result, we see that $W_2(\rho_m^{T/h}, \rho_\infty^{T/h}) = d_{T/h}$ satisfies
\begin{equation}
d_{T/h} \leq \sqrt{e^{CT} h^3 + 2(M+T+1)e^{CT} h} \leq C\sqrt{h}
\end{equation}
as long as $\delta \leq h^{3/2}$ (recall that $\delta =Cm^{-1/8}$, hence it is equivalent with $m > h^{-12}$), and so we are done.
\end{proof}

\section{Comparison principle and long-time behavior for gradient flow solutions}

\subsection{Comparison principle for the discrete-time solutions}
\label{subsec:comparison}

In the beginning of section \ref{sec:m_to_infty}, we have defined the discrete-time scheme \eqref{eq:def_jko} for the porous medium equation with drift \eqref{pme_drift}. Since the comparison principle for the viscosity solutions of \eqref{pme_drift} is well-known (see e.g. \cite{kl}), it is natural to ask whether the comparison principle holds for the discrete-time solutions generated by \eqref{eq:def_jko} as well. In this section we prove that this is indeed true, but the proof is quite different from the continuous case. In fact comparison principle-type results have been shown between discrete gradient flow solutions with $L^2$ distances, for instance in \cite{chambolle},  \cite{gk}, etc. The novelty in our result is that we address the discrete gradient flow solutions with $W_2$ distances, for which nonlocal perturbation arguments are necessary.

In order to define the scheme for two ordered initial data, we need to consider non-negative measures which do not necessarily integrate to 1. We denote by $\mathcal{P}_{2,A}(\mathbb{R}^d)$ the set of non-negative measures which integrate to $A>0$ and have finite second moment.  We also generalize the Wasserstein distance $W_2$ as follows:  For two regular measures $\rho_1, \rho_2 \in \mathcal{P}_{2,A}(\mathbb{R}^d)$, we define $W_2(\rho_1, \rho_2)$ as
$$W_2^2(\rho_1, \rho_2) :=\inf_{T\#\rho_1 = \rho_2}\int_{\mathbb{R}^d} |T(x)-x|^2 \rho_1(x) dx.$$

Next we state the comparison result. 
\begin{theorem} Let $\Phi$ satisfy \textup{(\textbf{A3})}.  For $2<m\leq \infty$, consider the two densities $\rho_{01}\in \mathcal{P}_{2,M_1}(\mathbb{R}^d)$, $\rho_{02} \in \mathcal{P}_{2,M_2}(\mathbb{R}^d)$ with the property $M_1\leq M_2$ and $\rho_{01} \leq \rho_{02}$ a.e. (In the case $m=\infty$, we require in addition that $\|\rho_{0i}\|_\infty \leq 1$ for $i=1,2$).   For  given $h>0$, let $\rho_1, \rho_2$ be the respective minimizers of the following schemes:
\begin{equation}
\rho_i := \underset{\rho \in \mathcal{P}_{2,M_i}(\mathbb{R}^d)} {\argmin}\mathcal{F}_i(\rho):= \underset{\rho \in \mathcal{P}_{2,M_i}(\mathbb{R}^d)} {\argmin} \left[ E_m[\rho] + \frac{1}{2h}W_2^2(\rho, \rho_{0i})\right] \quad\text{ for }i=1,2.
\label{def:rho_i}
\end{equation}
Then $\rho_1 \leq \rho_2$ almost everywhere.
\label{thm:comparison}
\end{theorem}

\begin{remark}
The proof of above theorem does not directly use the semi-convexity of $\Phi$, except to guarantee the existence of the $\rho_i$.
\end{remark}

Before we prove Theorem \ref{thm:comparison}, we first state and prove the following simple lemma, which  can be informally stated as follows: Given that $\rho_1$ is the minimizer for the discrete scheme in \eqref{def:rho_i} and $T_1$ is the optimal map between $\rho_{01}$ and $\rho_1$, if a part of $\rho_{01}$ is forced to be transferred by the map $T_1$, then $T_1$ is still the optimal map for the rest of $\rho_{01}$. 
\begin{lemma}
Let $2<m<\infty$, and let $h$ and $\rho_{01}$ be as given in Theorem \ref{thm:comparison}. We denote by $\rho_1$ the minimizer of $\mathcal{F}_1$ as given by \eqref{def:rho_i}, and let $T_1$ be the optimal mapping such that $T_1 \# \rho_{01} = \rho_1$.  Consider an arbitrary function $\eta: \mathbb{R}^d \to \mathbb{R}$ such that $0 \leq \eta(x) \leq 1$ for all $x\in \mathbb{R}^d$, and let $\varphi(x) := T_1 \# ((1-\eta)\rho_{01})$. Then $T_1 \# (\eta \rho_{01})$ minimizes 
$$ 
\mathcal{\tilde F}(\rho) := \int_{\mathbb{R}^d} \left( \frac{1}{m}(\varphi + \rho)^m + \rho \Phi \right) dx + \frac{1}{2h} W_2^2(\eta \rho_{01}, \rho)
 $$
among all $\rho \in \mathcal{P}_{2,\tilde M} (\mathbb{R}^d)$, where $\tilde M = \int_{\mathbb{R}^d} \eta \rho_{01} dx$.
\label{lemma:part}
\end{lemma}

\begin{proof}
Suppose that the minimum of $\mathcal{\tilde F}$ is achieved by another measure $\tilde \rho \in \mathcal{P}_{2,\tilde M} (\mathbb{R}^d)$, such that $\mathcal{\tilde F}(\tilde \rho) < \mathcal{\tilde F}(T_1 \# (\eta \rho_{01}))$. We denote by $\tilde T$ the optimal map such that $\tilde T \# (\eta \rho) = \tilde \rho$. The claim is then that we can find a better transfer plan of $\rho_{01}$ than $\rho_1$ in \eqref{def:rho_i}, yielding a contradiction. We construct the transfer plan as follows. 

First, we separate $\rho_{01}$ into two parts: $\eta \rho_{01}$ and $(1-\eta) \rho_{01}$. Then we use $T_1$ to push forward $ (1-\eta) \rho_{01}$, and use $\tilde T$ to push forward $\eta \rho_{01}$. The resulting measure would be equal to $\varphi + \tilde \rho$. Then it follows that 
\begin{equation}\label{ineq:wts_f}
\begin{split}
\mathcal{F}_1(\varphi+\tilde \rho)  &=  \int_{\mathbb{R}^d} \Big( \frac{1}{m}(\varphi + \tilde \rho)^m + (\varphi + \tilde\rho) \Phi \Big) dx + \frac{1}{2h} W_2^2(\rho_{01}, (\varphi + \tilde \rho))\\
&\leq   \int_{\mathbb{R}^d} \Big( \frac{1}{m}(\varphi + \tilde \rho)^m + (\varphi + \tilde\rho) \Phi \Big) dx + \frac{1}{2h} W_2^2(\eta \rho_{01}, \tilde \rho) + \frac{1}{2h} W_2^2((1-\eta)\rho_{01}, \varphi)\\
&= \mathcal{\tilde F}(\tilde \rho) + \int_{\mathbb{R}^d}  \varphi \Phi dx +  \frac{1}{2h} W_2^2((1-\eta)\rho_{01}, \varphi)\\
&< \mathcal{\tilde F}(T_1 \# (\eta \rho)) + \int_{\mathbb{R}^d}  \varphi \Phi dx +  \frac{1}{2h} W_2^2((1-\eta)\rho_{01}, \varphi)\\
&= \mathcal{F}_1(T_1 \# \rho) = \mathcal{F}_1(\rho_1),
\end{split}
\end{equation}
which contradicts the fact that $\rho_1$ is the minimizer of \eqref{def:rho_i} and so we are done.
\end{proof}

\noindent\begin{proof}[\textbf{Proof of  Theorem \ref{thm:comparison}}]

First we point out that  once we prove the comparison result for all \\$2<m<\infty$, it will be automatically true for the case $m=\infty$ as well, due to the one-step estimate in Proposition \ref{prop:one_step_m_infty}. To see this,  let us denote by $\rho_{i,m}$ the minimizer $\rho_i$ when the free energy is $E_m$.  Then Proposition \ref{prop:one_step_m_infty} gives us $\rho_{i,m} \to \rho_{i,\infty}$ as $m\to \infty$ in Wasserstein distance. Therefore if we know that $\rho_{1,m} \leq \rho_{2,m}$ a.e. for all $2<m<\infty$ then it directly follows that $\rho_{1,\infty} \leq \rho_{2,\infty}$ a.e.

Due to the above discussion, it suffices to prove the comparison principle for any fixed $m$ satisfying $2<m<\infty$.  Let $T_i$ denote the optimal map such that $T_i \# \rho_{0i} = \rho_i$ for $i=1,2$.  We prove by contradiction and suppose $\Omega := \{\rho_1>\rho_2\}$ has non-zero measure. We first claim that
 \begin{equation}
 |T_1^{-1}(\Omega) \backslash T_2^{-1}(\Omega)| > 0,
\label{ineq:omega} 
 \end{equation}
which directly follows from the inequality below:
\begin{equation}
\begin{split}
\int_{T_1^{-1}(\Omega)} \rho_{01} dx &= \int_{\Omega} \rho_1 dx \text{\quad(since $T_1\#\rho_{01} = \rho_1$)}\\
&> \int_{\Omega} \rho_2 dx \text{\quad(from the definition of $\Omega$ and the assumption that $|\Omega|>0$)}\\
&= \int_{T_2^{-1}(\Omega)} \rho_{02} dx  \text{\quad(since $T_2\#\rho_{02} = \rho_2$)}\\
&\geq \int_{T_2^{-1}(\Omega)} \rho_{01} dx \text{\quad(since $\rho_{01}\leq\rho_{02}$)}.
\end{split}
\end{equation}

Let $\Omega_\delta = \{x\in\mathbb{R}^d: \rho_1(x) > \rho_2(x)+\delta\}$, and let $A_\delta = \{x\in\mathbb{R}^d:\rho_1(T_1(x)) \leq \frac{1}{\delta}, \rho_2(T_2(x)) \leq \frac{1}{\delta}\}$. Since $\cup_{\delta>0}\Omega_\delta=\Omega$ and $\cup_{\delta>0}A_\delta =\mathbb{R}^d$,  \eqref{ineq:omega} yields that
\begin{equation}
 \Big|\big(T_1^{-1}(\Omega_\delta) \cap A_\delta \big) \backslash T_2^{-1}(\Omega)\Big| > 0 \hbox{ for sufficiently small } \delta>0.
\label{ineq:omega_delta}
\end{equation}
From now on we fix $\delta$ such that the above inequality is true, and denote 
$$B := \big(T_1^{-1}(\Omega_\delta) \cap A_\delta\big) \backslash T_2^{-1}(\Omega).$$

By definition of the set $B$, it immediately follows that $T_1$ maps $B$ into the set where $\rho_2 + \delta<\rho_1\leq\frac{1}{\delta}$, while $T_2$ maps $B$ into the set where $\rho_1 \leq \rho_2 \leq \frac{1}{\delta}$. (Note that these inequalities hold in the a.e. sense). These facts are illustrated in Figure \ref{fig:definition_B}.

\begin{figure}[h!]

\begin{tikzpicture}
\pgfmathsetmacro{\scale}{0.8}
\definecolor{rho1}{RGB}{150,0,188}
\definecolor{rho2}{RGB}{0,150,255}

\draw[black] (-6.5*\scale,0)--(0.5* \scale,0);

\draw [rho1, thick] (-5.3*\scale,0) [out = 80, in = 210] to (-4*\scale, 2.1*\scale) [out = 30, in = 180] to (-3*\scale,2.4*\scale) [out = 0, in = 120] to (-0.8*\scale,0);
\fill[pattern=north east lines wide, pattern color= rho1] (-4*\scale,0) -- (-4*\scale, 2.1*\scale) [out = 30, in = 180] to (-3*\scale,2.4*\scale)--  (-3*\scale,0) -- cycle;

\draw [rho2, thick] (-5.8*\scale,0) [out = 65, in = 210] to (-4*\scale, 2.6*\scale) [out = 30, in = 180] to (-3*\scale,2.9*\scale) [out = 0, in = 130] to (0,0);
\fill[pattern=north west lines wide, pattern color= rho2] (-4*\scale,0) -- (-4*\scale, 2.6*\scale) [out = 30, in = 180] to (-3*\scale,2.9*\scale)--  (-3*\scale,0) -- cycle;

\draw (-1*\scale,2.1*\scale) node[rho2] {\large{$\rho_{02}$}};
\draw (-1.7*\scale,1*\scale) node[rho1] {\large{$\rho_{01}$}};

\draw[dashed, gray] (-3*\scale,0)--(-3*\scale,2.9*\scale);
\draw[dashed, gray] (-4*\scale,0)--(-4*\scale,2.6*\scale);

\draw[thick] (-3*\scale,3.1*\scale) to [color = rho2, looseness = 0.7] (4*\scale,2.7*\scale);
   \draw[color=rho2,
    decoration={markings,mark=at position 1 with {\arrow[very thick]{>}}},
    postaction={decorate}
    ]
    (3.9*\scale,2.8*\scale) -- (4*\scale,2.7*\scale);
 \draw (0.5*\scale, 3.5*\scale) node[rho2] {{${T_2}_\# \rho_{02} = \rho_2$}};

\draw[thick] (-3*\scale,-0.5*\scale) to [color = rho1, bend right, looseness = 0.7] (6*\scale,-1*\scale);
   \draw[color=rho1,
    decoration={markings,mark=at position 1 with {\arrow[very thick]{>}}},
    postaction={decorate}
    ]
    (5.9*\scale,-1.05*\scale) -- (6*\scale,-1*\scale);
 \draw (1.5*\scale, -1.3*\scale) node[rho1] {{${T_1}_\# \rho_{01} = \rho_1$}};

\draw [decorate,decoration={brace, amplitude=6pt},xshift=0pt,yshift=-2pt]
(-3*\scale,0) -- (-4*\scale,0)node [black,midway,yshift=-12pt] {
$B$};
\draw[black] (2*\scale,0)--(9*\scale,0);
\draw [rho2, thick] (2.3*\scale,0) [out = 80, in = 210] to (3.5*\scale, 2.2*\scale) [out = 30, in = 170] to (4.8*\scale,2.5*\scale) [out = -10, in = 130] to (8.5*\scale,0);
\fill[pattern=north west lines wide, pattern color= rho2] (3.5*\scale,0) -- (3.5*\scale, 2.2*\scale) [out = 30, in = 170] to (4.8*\scale,2.5*\scale) --  (4.8*\scale,0) -- cycle;
\draw[dashed, gray] (3.5*\scale,0)--(3.5*\scale,2.2*\scale);
\draw[dashed, gray] (4.8*\scale,0)--(4.8*\scale,2.5*\scale);
\draw [decorate,decoration={brace, amplitude=6pt},xshift=0pt,yshift=-2pt]
(4.8*\scale,0) -- (3.5*\scale,0)node [black,midway,yshift=-13pt] {
$T_2(B)$};

\draw [rho1, thick] (3.7*\scale,0) [out = 60, in = 210] to (6*\scale, 2.4*\scale) [out = 30, in = 135] to (7*\scale,2.3*\scale) [out = -45, in = 100] to (8*\scale,0);
\fill[pattern=north east lines wide, pattern color= rho1] (6*\scale,0) -- (6*\scale, 2.4*\scale) [out = 30, in = 135] to (7*\scale,2.3*\scale) --  (7*\scale,0) -- cycle;
\draw[dashed, gray]  (6*\scale,0) -- (6*\scale, 2.4*\scale);
\draw[dashed, gray] (7*\scale,2.3*\scale) --  (7*\scale,0);
\draw [decorate,decoration={brace, amplitude=6pt},xshift=0pt,yshift=-2pt]
(7*\scale,0) -- (6*\scale,0)node [black,midway,yshift=-13pt] {
$T_1(B)$};
\draw (8*\scale,1.5*\scale) node[rho1] {\large{$\rho_{1}$}};
\draw (2.9*\scale,2.1*\scale) node[rho2] {\large{$\rho_{2}$}};
\end{tikzpicture}

\caption{Illustration of the set $B, T_1(B)$ and $T_2(B)$. Recall that $T_i$ is the optimal map between $\rho_{0i}$ and $\rho_i$ for $i=1,2$. Moreover, the set $B$ is chosen such that $T_1(B) \subset \{ \frac{1}{\delta} \geq \rho_1 > \rho_2+\delta\}$, while $T_2(B) \subset \{ \rho_1 \leq \rho_2 \leq \frac{1}{\delta}\}$.\label{fig:definition_B}}
\end{figure}
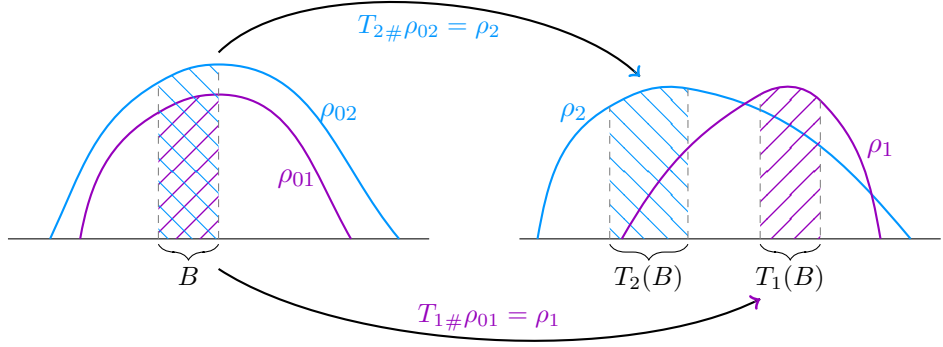
Let $\rho_\epsilon := \epsilon\rho_{01}\chi_{B}$, where $0<\epsilon\ll \delta$ is a sufficiently small number to be determined later. Let $\varphi_1 := T_1 \# (\rho_{01} - \rho_{\epsilon})= (1-\epsilon)\rho_1 \chi_{T_1(B)}$, then by applying Lemma~\ref{lemma:part} to the optimal plan $T_1$ in comparison to $T_2$, and subtracting $\frac{1}{m}\varphi_1^m$ on both sides, we arrive at the following inequality:
\begin{equation}
\begin{split}
&\int_{\mathbb{R}^d} \left( \frac{1}{m} (\varphi_1 + T_1 \# \rho_\epsilon)^m - \frac{1}{m} \varphi_1^m \right) dx + E[T_1]\\
\leq & \int_{\mathbb{R}^d} \left( \frac{1}{m} (\varphi_1 + T_2 \# \rho_\epsilon)^m - \frac{1}{m} \varphi_1^m \right) dx + E[T_2],
\end{split}
\label{ineq:T1_2}
\end{equation}

where
$$
E[T_i] := \int_{\mathbb{R}^d} \Big((T_i \# \rho_\epsilon)\Phi + \frac{1}{2h} |T_i(x) - x|^2 \rho_\epsilon \Big) dx, ~~i = 1,2.
$$

Next we state a simple algebraic inequality without proof. For all real numbers $a$ and $b$ satisfying $0<b< a< \frac{1}{\delta}$ and $m>2$, we have
\begin{equation}
a^{m-1}b \leq \frac{1}{m} (a+b)^m - \frac{1}{m} a^m \leq a^{m-1}b + Cb^2,
\label{ineq:algebraic}
\end{equation}
where the constant $C$ only depends on  $m$ and $\delta$. Using \eqref{ineq:algebraic},  \eqref{ineq:T1_2} yields that
\begin{equation}
\begin{split}
 \int_{\mathbb{R}^d} \varphi_1^{m-1} (T_1 \# \rho_\epsilon) dx + E[T_1] 
&\leq  \int_{\mathbb{R}^d} \varphi_1^{m-1} (T_2 \# \rho_\epsilon) dx + E[T_2] + \int_{T_2(B)} C(m,\delta) \epsilon^2 \rho_2^2 dx
\\&\leq \int_{\mathbb{R}^d} \varphi_1^{m-1} (T_2 \# \rho_\epsilon) dx + E[T_2] + C(m,\delta) \epsilon^2 \text{~~~ (since $\rho_2 \leq \frac{1}{\delta}$ in $T_2(B)$)}\\
 & \leq \int_{\mathbb{R}^d} \rho_1^{m-1} (T_2 \# \rho_\epsilon) dx + E[T_2] + C(m,\delta) \epsilon^2  \text{~~ (since $\varphi_1 \leq \rho_1$)}\\
 & \leq \int_{\mathbb{R}^d} \rho_2^{m-1} (T_2 \# \rho_\epsilon) dx + E[T_2] +C(m,\delta) \epsilon^2   \end{split}
 \label{ineq:T1_3}
 \end{equation}
where the last inequality holds since $\rho_1 \leq \rho_2$ in $T_2(B)$, and $\text{ supp} (T_2 \# \rho_\epsilon) \subset T_2(B)$.

Similarly, we define $\varphi_2 := T_2 \# (\rho_{02} - \rho_{\epsilon})$, and note that $\rho_{2}$ is the minimizer to \eqref{def:rho_i}. We then apply Lemma~\ref{lemma:part} to the optimal plan $T_2$ in comparison to $T_1$, and an argument parallel to that above yields the following inequality for $\varphi_2$:
\begin{equation}
\begin{split}
 \int_{\mathbb{R}^d} \varphi_2^{m-1} (T_2 \# \rho_\epsilon) dx + E[T_2] &\leq \int_{\mathbb{R}^d} \varphi_2^{m-1} (T_1 \# \rho_\epsilon) dx + E[T_1] + C(m,\delta) \epsilon^2\\
 &\leq \int_{\mathbb{R}^d} \rho_2^{m-1} (T_1 \# \rho_\epsilon) dx + E[T_1] + C(m,\delta) \epsilon^2 \text{~~ (since $\varphi_2 \leq \rho_2$)} 
 \end{split}
 \label{ineq:T2}
 \end{equation}
  Note that in the set $T_2(B)$, $\rho_2$ is bounded above by $\frac{1}{\delta}$, hence $\varphi_2$ is just smaller than $\rho_2$ by order $\epsilon$ in this set, namely $\rho_2 < \varphi_2 + \epsilon/\delta$ in $T_2(B)$. Combining this with the fact that the integral of $T_2 \# \rho_\epsilon$ is also of order $\epsilon$, and the assumption that $m>2$, we have the following: 
\begin{equation}
\int_{\mathbb{R}^d} \rho_2^{m-1} (T_2 \# \rho_\epsilon) dx \leq  \int_{\mathbb{R}^d} \varphi_2^{m-1} (T_2 \# \rho_\epsilon) dx+ C(m,\delta) \epsilon^2.
\label{ineq:link}
\end{equation}
\eqref{ineq:link} provides us a link between the RHS of \eqref{ineq:T1_3} and the LHS of \eqref{ineq:T2}, and so we arrive at
\begin{equation}
 \int_{\mathbb{R}^d} \varphi_1^{m-1} (T_1 \# \rho_\epsilon) dx \leq 
 \int_{\mathbb{R}^d} \rho_2^{m-1} (T_1 \# \rho_\epsilon) dx + C(m,\delta) \epsilon^2. 
 \label{ineq:combination}
 \end{equation}

Next we show that \eqref{ineq:combination} leads to a contradiction if $\epsilon$ is chosen to be small enough. First, recall that $\phi_1 = (1-\epsilon) \rho_1 \chi_{T_1(B)}$, and $\rho_1 > \rho_2+\delta$ in $T_1(B)$. Hence if we let $\epsilon$ be sufficiently small, we would have $\phi_1 > \rho_2 + \frac{\delta}{2}$ in $T_1(B)$.  Then we have
 \begin{equation}
 \begin{split}
 \int_{\mathbb{R}^d} \varphi_1^{m-1} (T_1 \# \rho_\epsilon) dx &\geq 
 \int_{\mathbb{R}^d} (\rho_2+\frac{\delta}{2})^{m-1} (T_1 \# \rho_\epsilon) dx\\
&\geq  \int_{\mathbb{R}^d} \left(\rho_2^{m-1}+(\frac{\delta}{2})^{m-1}\right)(T_1 \# \rho_\epsilon) dx\\
 &\geq \int_{\mathbb{R}^d}  \rho_2^{m-1} (T_1 \# \rho_\epsilon) dx + \epsilon (\frac{\delta}{2})^{m-1} \|\rho_1\|_{L^1(B)},
 \end{split}
 \end{equation}
 which contradicts \eqref{ineq:combination} when we fix $\delta$ and let $\epsilon$ be sufficiently small.  This concludes the proof.
 \end{proof}

\begin{remark}
By sending the time step $h\to 0$, the comparison principle for discrete solutions immediately leads to a comparison principle for gradient flow solutions. Also, although we only prove the comparison principle for the energy $\int \rho^m dx$ with $2<m\leq \infty$, the proof can indeed be easily extended for $1<m\leq \infty$, and also the case when the entropy part is given by$\int \rho \log \rho dx$. 
\end{remark}

\subsection{Confinement result and long-time behavior.}
In this subsection, we show some applications of the comparison principle for discrete JKO solutions. 
The first application is the following confinement result for discrete solutions (hence continuous gradient flow solutions as well), given that $\Phi\to+\infty$ as $|x|\to\infty$. 

\begin{corollary}\label{confinement}
Let $2<m\leq \infty$ and let $\Phi(x)$ satisfy \textup{(\textbf{A3})} and the additional assumption that $\lim_{|x|\to\infty}\Phi(x) =+ \infty$. Assume the initial data $\rho_0 \in L^\infty(\mathbb{R}^d)$ has compact support, and if $m=\infty$ we assume in addition that $\|\rho_0\|_\infty \leq 1$. Then the support for the discrete solution $\rho_{m,h}^n$ will stay bounded for all $n$, where the bound of the support does not depend on $n$ or $h$.
\end{corollary}

\begin{proof}
\emph{$\circ$ Case 1: $2<m<\infty$}. For any $A>0$, let us look for the global minimizer $\rho_A$ of the energy $E_m$ as defined by \eqref{approx} among $\mathcal{P}_{2,A}(\mathbb{R}^d)$. Due to \cite[Lemma 6]{cjmtu}, the global minimizer $\rho_A$ is given by 
$$\rho_A= \Big(\frac{m-1}{m}(C_A-\Phi(x))_+\Big)^{\tfrac{1}{m-1}},$$
where $C_A$ is chosen such that the total mass of $\rho_A$ is equal to $A$. Observe that for any $A>0$, such $\rho_A$ is also a stationary solution for the discrete JKO scheme, and it has a compact support. 

Therefore for any $\rho_0\in L^\infty(\mathbb{R}^d)$ with compact support, one can choose $A$ to be sufficiently large such that $\rho_0 \leq \rho_A$ a.e. Then one can apply Theorem \ref{thm:comparison} and obtain that $\rho_{m,h}^n \leq \rho_A$ a.e., hence the support of $\rho_{m,h}^n$ stays within the support of $\rho_A$ for all time steps.

\emph{$\circ$ Case 2: $m=\infty$}. 
In this case, we first point out that for any mass size $A>0$, the global minimizer of $E_\infty$ among $\mathcal{P}_{2,A}$ is given by some characteristic function $\chi_{S_A}$, where $S_A$ is the level set of the function $\Phi$, i.e.
$$S_A = \{x\in\mathbb{R}^d: \Phi(x) \geq C_A\},$$
and $C_A$ is chosen so that $\chi_{S_A}$ has mass $A$. Moreover, since $\chi_{S_A}$ is the global minimizer of $E_\infty$, it must be a stationary solution as well. \\[0.3cm]
Recall that $\rho_0$ has compact support, $\|\rho_0\|_\infty \leq 1$, and $\Phi(x) \to \infty$ as $|x|\to\infty$. Therefore if we let $A$ be sufficiently large, we will have $\text{supp}~\rho_0 \subset S_A$, which implies that $\rho_0 \leq \chi_{S_A}$. Since $\chi_{S_A}$ is a stationary solution, the comparison result in Theorem \ref{thm:comparison} immediately implies that $\text{supp}~\rho_{\infty, h}^n \subset S_A$ for all $n$ and $h$, and we are done.
\end{proof}

Lastly we briefly discuss the long time behavior of the gradient flow solution $\rho_m$ for $2<m\leq \infty$, when $\Phi$  is strictly convex and bounded below in $\mathbb{R}^d$.  In this case, one can easily obtain that the global minimizer for $E_\infty$ in $\mathcal{P}_2(\mathbb{R}^d)$ is $\rho_S := \chi_\mathcal{O}$, where $\mathcal{O} = \{x\in\mathbb{R}^d: \Phi(x) \leq C\},$ and $C$ is chosen such that $\chi_\mathcal{O}$ has mass 1. 
\begin{theorem}\label{long_time_convergence} 
Let $2<m\leq \infty$.  Let $\Phi$ be strictly convex and satisfy \textup{(\textbf{A2})} and \textup{(\textbf{A3')}}.  Assume the initial data $\rho_0 \in \mathcal{P}_2(\mathbb{R}^d)$ has compact support, and in addition satisfies $\|\rho_0\|_\infty \leq 1$ in the case $m=\infty$. For $2<m\leq \infty$, let $\rho_{m}$ be given as the gradient flow for $E_m$ with initial data $\rho_0$, as defined in Theorem \ref{thm:collection}(b). Then as $t\to\infty$, $\rho_{m}(\cdot,t)$ converges to the unique global minimizer $\rho_S$ of $E_m$ exponentially fast in 2-Wasserstein distance.
\end{theorem}
\begin{proof}
If $\Phi$ is uniformly convex in $\mathbb{R}^d$, then there exists some $\lambda>0$, such that $D^2 \Phi(x) \geq \lambda I$ for all $x\in \mathbb{R}^d$. In this case we can directly apply the contraction result in  Theorem \ref{ags}(c) between $\rho_\infty(x, t)$ and $\rho_S(x)$ (where $\rho_S$ is the global minimizer for the free energy $E_m[\rho]$ in $\mathcal{P}_2(\mathbb{R}^d)$), which gives
$$ W_2(\rho_\infty(\cdot, t), \rho_S(\cdot)) \leq W_2(\rho_0, \rho_S) e^{-\lambda t},$$
and hence the 2-Wasserstein distance between $\rho_\infty(x,t)$ and $\rho_S(x)$ decays exponentially fast in $t$.\\
On the other hand, if $\Phi$ is strictly convex in $\mathbb{R}^d$ but not uniformly convex, we will make use of the confinement result in Corollary \ref{confinement}. As long as $\rho_0$ is compactly supported,  the proof of Corollary \ref{confinement} shows the support of $\rho_\infty(\cdot, t)$
will stay in some compact set $\mathcal{O}_A$ for all time, and indeed one can find an $\mathcal{O}_A$ such that it is independent of $m$ for all $2<m\leq \infty$. This confinement result allows us to apply the contraction result in Theorem \ref{ags} (c), which gives that
$$ W_2(\rho_\infty(\cdot, t), \rho_S(\cdot)) \leq W_2(\rho_0, \rho_S) e^{-\tilde \lambda t},$$
where $\tilde \lambda = \inf\{\lambda: D^2 \Phi(x) \geq \lambda I \text{ for all }x\in \mathcal{O}_A\}$ is a strictly positive constant depending on $\rho_0$ and $\Phi$. 
\end{proof}

Finally we remark that for finite $m$ and $\rho_m$, the corresponding result is shown in \cite{cjmtu}, where they use entropy dissipation methods. We suspect the convergence rate to be exponential in stronger norms instead of Wasserstein distance, but this issue is not pursued here.

\appendix

\section{ Constructing a (PME-D)$_m$ subsolution}
\begin{lemma}\label{appendix} Fix $\epsilon > 0$, $m > 0$, a number $\gamma$, a point $x'$, and vector $\vec n$.  Then if $\gamma > \nabla \Phi(x') \cdot \vec n$, there exists a positive constant $\eta$ depending on $\epsilon$ so that we can construct a classical subsolution $S$ of $(PME-D)_m$  in  $E_\eta := B_{\eta}(x') \times [-\eta, \eta]$ with $(x',0)$ on its free boundary with outward normal $\vec n$ , which moves with normal velocity $\gamma$. Further, $S$ will be an ``almost" supersolution near $(x',0)$ in the following sense:
\[\gamma \ge  |\nabla S| + \nabla S \cdot \nabla \Phi -\epsilon \hbox{ at } (x',0). \]

\end{lemma}

\begin{proof}
Recall that  the Barenblatt profiles are given by
\[B(x, t; \tau, C) = \frac{(C(t+\tau)^{2\lambda} - K|x-x_0|^2)_+}{(t+\tau)}\]
where $C > 0$ and $\lambda = ((m-1)n+2)^{-1}$, $K = \lambda / 2$, and they are solutions of (PME)$_m$.  $C$ and $\tau$ are parameters that control the free boundary speed and initial support.

Now we change variables so that $\vec n$ is colinear with $x'$, and take $x_0 = 0$ .  Then it suffices to take $x_0 = 0$. We start with $B(x,t)$: a Barenblatt solution with initial support $B_{R}(0)$, and initial free boundary advancement speed $\xi$ where $R$ and $\xi$ will be determined later.  We fix $r(t) = \mu - \nu t$, with $\mu,\nu$ as yet unspecified.  Then we define 
\[\tilde{S}(x,t) = \sup_{y \in B_{r(t)}(x)} B(y,t) = B((1-r(t)/|x|)x,t) \hbox { in } E_{\eta},
 \]
where $\eta$ is for now much smaller than $R/2$.     

Note that 
\[\tilde{S}_t = B_t - r'(t) \nabla B \cdot \frac{x}{|x|} = B_t + r'(t) |\nabla B| = B_t - \nu |\nabla B|\]
Moreover in $E_\eta$, since $\eta < R/2$, then $E_\eta$ is bounded away from the origin and we get $1/|x| \le 2/R$.  Thus we find that
\begin{eqnarray*}
  \frac{\partial \tilde{S}}{\partial x_j} & = & 
   \frac{\partial B}{\partial x_j} +\mu  |\nabla B|O( 1/R).
\end{eqnarray*}
Thus $\nabla \tilde{S} = \nabla B + O(\mu)$ and since $|\nabla B|$ does not vary fast in $E$, we can repeat and find $\Delta \tilde{S} = \Delta B + O(\mu)$.  Using that $B$ is a $(PME)_m$ solution then gives 
\[\tilde{S}_t = B_t - \nu |\nabla B| = (m-1)\tilde{S}\Delta \tilde{S}  + |\nabla \tilde{S}|^2 -\nu |\nabla \tilde{S}| + O(\mu)\]
Next let us define
\[S(x,t) = \tilde{S}(x + \vec{b}t,t), \hbox{ where }\vec{b} = \nabla \Phi(x',0). \]
Then $S_t = \tilde{S}_t + \nabla \tilde{S} \cdot \vec{b}$, and one can conclude that 
\begin{eqnarray*}
  S_t & = & (m-1)S(\Delta S + \Delta \Phi) + |\nabla S|^2  + \nabla S \cdot \nabla \Phi - (m-1)S \Delta \Phi -\nu |\nabla S|+O(\mu)  
\end{eqnarray*}
Now in $E_\eta$, $S(x,t) \le 2\eta \sup_{E_\eta}|\nabla S| = O(\eta)$, so $(m-1)S \Delta \Phi = O(\eta)$.  Therefore 
\[S_t  = (m-1)S(\Delta S + \Delta \Phi) + |\nabla S|^2  +  r'(t) |\nabla S| + \nabla S \cdot \nabla \Phi +O(\mu) + O(\eta) \]

At this point we have to start picking our parameters carefully.  First, we can assume that $\epsilon$ is small enough so that $\epsilon < \inf_{E_\eta\cap \{S>0\}} |\nabla S| / 6$ for some small value of $\eta$. 
 Then we take 
\[\nu = \epsilon / 3,\;\; \xi = \gamma - \vec b \cdot x' + \nu > 0\]
Now we take $\eta, \mu$ small enough so that in $E_\eta$, 
\[|O(\mu) + O(\eta)|< \epsilon \inf_{E_\eta\cap \{S>0\}}|\nabla S|/3\]
and we set $R = |x'| - \mu = 1 - \mu$.  Now we refine $\eta$ so that
\[\sup_{E_\eta\cap \{S>0\}} |\nabla S| -  \inf_{E_\eta\cap \{S>0\}} |\nabla S| < \epsilon \inf_{E_\eta\cap \{S>0\}} |\nabla S|
\]
Then our choice of $\nu$ gives us the estimates
\[\nu |\nabla S| \ge \frac{\epsilon \inf_{E_\eta\cap \{S>0\}} |\nabla S|}{3 }\]
while also
\[\nu |\nabla S| \le \frac{\epsilon \sup_{E_\eta\cap \{S>0\}} |\nabla S|}{3} < \frac{\epsilon(\inf_{E_\eta\cap \{S>0\}} |\nabla S| + \epsilon)}{3} = \epsilon \inf_{E_\eta\cap \{S>0\}} |\nabla S| /3  + \epsilon^2/3 \le \epsilon \inf_{E_\eta\cap \{S>0\}} |\nabla S| / 2\]
where we used our assumption on $\epsilon$ small.
Thus we find that
\[ -\epsilon \inf_{E_\eta\cap \{S>0\}} |\nabla S| \le r'(t) |\nabla S| + O(\eta) + O(r(t)) \le 0\]
and so finally
\[(m-1)S(\Delta S + \Delta \Phi) + |\nabla S|^2 + \nabla S \cdot \nabla \Phi - \epsilon  \inf_{E_\eta\cap \{S>0\}} |\nabla S|\le S_t  \le (m-1)S(\Delta S + \Delta \Phi) + |\nabla S|^2 + \nabla S \cdot \nabla \Phi  \]
Then we are done, since it is clear that $(x',0)$ is on the free boundary of $S$ and the free boundary has initial velocity $\gamma$.

\end{proof}

\section{The support of (PME-D)$_m$ solutions have bounded jumps}\label{AppendixB}

\begin{theorem}\label{pmeBddSpeed}
Suppose $u_m$ is a solution to (PME-D)$_m$ in $\R^d$. Then for $K > 0$, there exist constants $r_{\max}, T>0$ only depending on $K, d$, and $\Phi$ near $x'$  such that the following holds for any $r_0<r_{\max}$:
Suppose $u_m(\cdot,t') = 0$ in $B_{r_0}(x')$ and $u_m \le K$ on the parabolic boundary of $B_{2r_0}(x') \times [t', t'+T]$.  Then we have that $u_m = 0$ in $B_{r_0/4}(x') \times[t', t' + T]$.
\end{theorem}

\begin{proof}
1.  We may assume that $(x',t') = 0$. Now to prove this theorem, first we construct a supersolution of (PME-D)$_m$ in $B_{2r_0}(0)\times [0, T]$ where $T$ is yet to be determined. We start by constructing $u(r)$ on $B_{2r_0}(0)$ satisfying $u(r_0) = K$ and $u(r) = 0$ for $r \le r_0/2$.  We take $r_0$ small enough so that $\sup_{B_{4r_0}(0)} |\nabla \Phi(x) - \nabla \Phi(0)| < 1$.  Define $\alpha = \sup_{B_{4r_0}(0)} \Delta \Phi(x)$.  Then we solve that $\Delta u = -\alpha$ in the annulus $B_{2r_0}(0)\backslash B_{r_0/2}(0)$. This yields $u$ given by 
$$
u= \left\{\begin{array}{lr} \frac{C}{r^{d-2}} - \frac{\alpha r^2}{2d} + D & \mbox{d $\ne$ 2} \\ C \ln r - \alpha r^2/4 + D &\mbox{N $=$ 2} \end{array} \right.
$$
We proceed assuming $d >2$; the $d=2$ case is similar.  We choose $C$ and $D$ so that $u(r_0) = K, u(r_0/2) = 0$:
\begin{eqnarray*}
  C = -r_0^{d-2}\frac{K+3 \alpha r_0^2/8d}{2^{d-2}-1}, \quad
D = K+\alpha r_0^2/2d +\frac{K+3\alpha r_0^2/8d}{2^{d-2}-1}
\end{eqnarray*}
By taking derivatives it can be seen that $u$ has the largest derivative at $r_0/2$, and we then estimate:
\[u'(r_0/2)  \le C(d,K)/r_0 + C(d,\Phi) d r_0^2
\]
Further, we notice that $u' \ge O(1/r_0^{d-1})$ and so if we take $r_0$ small enough we have $u'(r) \ge 1$ when $r_0 / 2 \le r \le 4 r_0$.  Then this entails that $u(r) \ge u(r_0) = K$ if $r \in [r_0, 4r_0]$, and we are done finding $r_{\max}$ which is the largest value of $r_0$ that makes the desired estimates hold.

2.  Now let us define 
\[\tilde{u}(r,t) := u(R(t) r)\]
where $R(t)$ is a function to be determined with $R(0) = 1$, $1 \le R(t) \le 3/2$.  Then by construction of $u$,
\[\Delta \tilde{u} = R(t)^2 (\Delta u)(R(t)r) = -\alpha R(t)^2 \le -\alpha \le 0 \]
Further straightforward computation yields that $\tilde{u}_t \ge 2|\nabla\tilde{u}|^2$ holds if 
\[\frac{R'(t)}{ R(t)^2} > 2\frac{u'(R(t)r)}{r}\]
To this end, let us choose $R(t) = 1/(1-Lt)$, where 
\[L:=C(N,\Phi, K)/r_0^2 = \frac{8}{r_0} \sup_{r \in [r_0/2,4r_0]}  u'(r) \ge 2\sup_{r \in [r_0/2,8r_0/3]} \frac{u'(R(t)r)}{r}.\]

3.   Lastly we define
\[v(x,t) := \tilde u(x + \vec b t, t)\]
where $\vec b = \nabla \Phi(0,0)$.  We claim that $v$ is a (PME-D)$_m$ supersolution in $B_{2r_0}(0) \times [0,T(d,\Phi,K)]$ for any choice of $m$. To see this, note that
\begin{align*}
  v_t - \vec b \cdot \nabla v& \ge 2 |\nabla v|^2 \ge (m-1)v (\Delta v + \Delta \Phi) + |\nabla v|^2 + |\nabla v|^2.
\end{align*}
Now  if $r_0< r_{\max}$, we have that
\[|\nabla v|^2 - \nabla v \cdot (\nabla \Phi - \vec b) \ge |\nabla v|(|\nabla v| - |\vec b - \nabla \Phi| \ge |\nabla v|(|\nabla v| - 1)\]
But we know that $|\nabla v| = u'(r) \ge R(t) u' \ge u' \ge 1$, so the above quantity is positive.  Thus
\[ v_t \ge (m-1)v (\Delta v + \Delta \Phi) + |\nabla v|^2 + \nabla v \cdot \nabla \Phi
\]
and since $v_t \ge |\nabla v|^2 + \vec b \cdot \nabla v$ in general this holds at the boundary too.  Thus $v$ is a classical free boundary supersolution, and so by Lemma 2.6 in \cite{kl} a viscosity supersolution.

Lastly, $T$ is chosen so that both $|\vec b| T \le 2r_0/3$ and $T < 1/(3L)$, and lastly so that $T < r_0 / (12 |\vec b|)$.  The first condition ensures that the bounds on $\tilde u$ in $B_{8r_0/3}$ hold for $v$ in $B_{2r_0}$ and the second ensures that $R(t) \le 3/2$.  The last one ensures that if $|x| \le r_0/4$, $|x + \vec b t | \le |x | + r_0 / 12 \le r_0 / 3 \le r_0 / 2R(t)$ and hence 
\[v(x,t) = u(R(t)(x+\vec b t),t) = 0 \hbox{ in } |x|\leq r_0/4.\]
Now by construction $u_m \le v$ on the parabolic boundary of $B_{2r_0}(0)\times[0,T(d,\Phi,K)]$, so we can apply a comparison principle, Theorem 2.25 in \cite{kl}, to find that 
\[u_m \le v \mbox{ in } B_{2r_0}(0) \times [0,T(d,\Phi,K)]\]
Then observing the properties of $v$, we are done.
\end{proof}

\section{Some prior results on gradient flows }
\label{AppendixC}In this part of appendix, we state some results from \cite{ags}, concerning the existence and uniqueness of the discrete solution $\rho_{m,h}^n$ as defined in \eqref{eq:def_jko}, and the convergence as the time step $h\to 0$. 

The key step leading to these results is the $\lambda$-convexity of $E_m[\rho]$ for all $1<m\leq \infty$ along the generalized geodesics. Thus we first digress a little bit to state some definition and results from the optimal transport theory (see e.g. section 9.2 in \cite{ags}). Recall that $\mu\in\mathcal{P}_2(\mathbb{R}^d)$ is regular if $\mu\in L^p(\mathbb{R}^d)$ with some $p>1$. 

\begin{definition}[generalized geodesics] Let the reference measure $\mu^1 \in \mathcal{P}_2(\mathbb{R}^d)$ be regular. Let $\mu^2, \mu^3 \in \mathcal{P}_2(\mathbb{R}^d)$; then we can find two optimal transport maps $\boldsymbol{t}^2$ and $\boldsymbol{t}^3$ such that $\boldsymbol{t}^i_\# \mu^1 = \mu^i$ and $W_2^2(\mu_1, \mu^i) = \int_{\mathbb{R}^d} |\boldsymbol{t}^i(x)-x|^2 d\mu^1(x)$ for $i=2,3$. The generalized geodesics joining $\mu^2$ to $\mu^3$ (with base $\mu^1$) is defined as
\begin{equation}
\mu_t^{2\to 3} = (\boldsymbol{t}_t^{2\to 3})_\# \mu^1 \quad\text{where}\quad \boldsymbol{t}_t^{2\to 3} := (1-t) \boldsymbol{t}^2 + t\hspace{0.5mm}\boldsymbol{t}^3, \quad t\in[0,1].
\label{def_general_geo}
\end{equation}
\label{def:geodesics}
\end{definition}
 Using the notion of generalized geodesics, one can define a notion of semi-convexity (or {\it $\lambda$-convexity}) for energy functionals on $\mathcal{P}_2(\mathbb{R}^d)$:

\begin{definition}[$\lambda$-convexity along generalized geodesics]
Given $\lambda \in \mathbb{R}$, a functional $E$ is called \emph{$\lambda$-convex along the generalized geodesics} if for any $\mu_1, \mu_2$ and $\mu_3$ satisfying the conditions in Definition \ref{def_general_geo}, the following inequality holds
$$E[\mu_t^{2\to 3}] \leq \mathcal (1-t) E[\mu_2] + t E[\mu_3] - \frac{\lambda}{2} t(1-t) \int_{\mathbb{R}^d} |\boldsymbol{t}^2 - \boldsymbol{t}^3|^2 d\mu_1 \text{ for all }0\leq t\leq 1,$$
where $\mu_t^{2\to 3}$, $\boldsymbol{t}^2$ and $\boldsymbol{t}^3$ are as defined in Definition \ref{def:geodesics}.
\end{definition}

The following Lemma is a direct consequence of \cite[Sec 9.3]{ags}, which says that as long as $\Phi$ is semi-convex, the functional $E_m$ would be convex for all $1<m\leq \infty$. Since the case $m=\infty$ is not directly covered in the book, we provide a short proof below for the sake of completeness.

\begin{lemma}[\cite{ags}]
Let $\Phi$ satisfy \textup{(\textbf{A3})} , and let $E_m:\mathcal{P}_2(\mathbb{R}^d)\to \R$ be as defined as in \eqref{def:E_m}. Then $E_m$ is  $\lambda$-convex along general geodesics for all $1<m\leq \infty$.
\end{lemma}
\begin{proof}
Due to $(\textbf{A3})$, Proposition 9.3.2 of \cite{ags} gives the $\lambda$-convexity of the functional $\int_{\mathbb{R}^d} \rho \Phi dx$ along generalized geodesics. 

For a finite $m>1$, let $S_m$ be given by \eqref{def:S_m}. One can directly apply Proposition 9.3.9 in \cite{ags} to obtain the convexity of $\mathcal{S}_m$ along generalized geodesics. Since the sum of two $\lambda$-convex functionals is still $\lambda$ convex, we obtain the $\lambda$-convexity of $E_m$ for any finite $m>1$.

It remains to check that the functional $\mathcal{S}_\infty$ defined in \eqref{def:S_infty} is also $\lambda$-convex along generalized geodesics. To do this, let $\mu^i$, $i=1,2,3$ be as given in Definition \ref{def_general_geo}. It suffices to show that if $\|\mu^i\|_{L^{\infty}} \leq 1$ for $i=2,3$, then $\|\mu_t^{2\to 3} \|_{L^\infty} \leq 1$ for all $0<t<1$ as well.  Note that due to the $\lambda$-convexity of $S_m$ for all $m>1$, we obtain 
$$\|\mu_t^{2\to 3}\|_{L^m} \leq \min\{\|\mu^2\|_{L^m}, \|\mu^3\|_{L^m}\} \text{~~for all }m>1,$$
and sending $m\to \infty$ immediately yields the desired result.
\end{proof}

Once we have the $\lambda$-convexity of $E_m$, Lemma 9.2.7 in \cite{ags} guarantees that the Assumption 4.0.1 in \cite{ags} is satisfied, which leads to the existence, uniqueness and convergence results in Theorem \ref{ags}.

\end{document}